\documentclass{article}
\usepackage{amsmath,amssymb,amsthm}
\usepackage{cite}

\begin{document}

\begin{center}
{\bf Linear differential  systems over the
quaternion skew field.}\end{center}

\begin{center}{ \textbf{\emph{Ivan I. Kyrchei} }\footnote{kyrchei@online.ua,\\ Pidstrygach Institute for Applied Problems of Mechanics and Mathematics of NAS of Ukraine,
 Ukraine} }\end{center}

\begin{abstract} A  basic theory on the  first order right and left linear quaternion differential systems (LQDS) is given systematic in this paper.  To proceed the theory of LQDS we adopt the theory of column-row determinants  recently  introduced by the author.
In this paper, the algebraic structure of their general solutions are established.  Determinantal representations of solutions of  systems with constant coefficient matrices and sources vectors are obtained in both cases when coefficient matrices are invertible and singular. In the  last case, we use determinantal representations of the quaternion Drazin inverse   within the framework of the theory of column-row determinants.

 Numerical examples to illustrate the main results are given.
\end{abstract}



\section{Introduction}
\newtheorem{Corollary}{Corollary}[section]
\newtheorem{theorem}{Theorem}[section]
\newtheorem{lemma}{Lemma}[section]
 \theoremstyle{definition}
\newtheorem{ex}{Example}[section]
\newtheorem{definition}{Definition}[section]
\theoremstyle{remark}
\newtheorem{remark}{Remark}[section]
\newcommand{\rank}{\mathop{\rm rank}\nolimits}
\newtheorem{proposition}{Proposition}[section]

Recently, considerable attention has been  paid to  the quaternion differential equations  with real variables (QDEs) which
have many applications in fluid mechanics (e.g. \cite{gib,rou}), quantum mechanics (see e.g. \cite{adl,leo1}),  Frenet frame in differential geometry \cite{han}, the attitude orientation  and spatial rigid body dynamics \cite{udw,gup},
etc.
Though QDEs have  many applications,  theoretical aspects of QDEs has been considered in  few papers.

Leo and Ducati \cite{leo2}  solved some simple
second order quaternionic differential equations. Campos and Mawhin \cite{cam}  studied  the existence of periodic
solutions  for the quaternionic Riccati equation.

 \.{Z}o{\l}\c{a}dek  \cite{zol} given the complete description
of dynamics of the Riccati equation.
Wilczynski proved the existence of two periodic solutions of quaternionic Riccati equations in \cite{wil1}, and considered some sufficient conditions for the existence of at least
one periodic solution of the quaternionic polynomial equations in\cite{wil2,wil3}.
Gasull et al. \cite{gas} proved the existence of periodic orbits, homoclinic loops, invariant
tori for a one-dimensional quaternionic autonomous homogeneous differential equation.
Zhang \cite{zha} studied the global structure of the quaternion Bernoulli
equations. Recently, Cai and Kou \cite{cai} achieved the Laplace transform approach to solve the linear quaternion differential equations.

Even fewer papers are devoted to systems of the linear quaternion differential equations.
We note the three papers from the arXiv \cite{kou1,kou2,kou3}, where the authors studied the basic theory of
LQDS such as the fundamental matrix,  the algebraic structure of solutions, etc.
Through the non-commutativity of the quaternion algebra, the construction  of the basic theory of  linear systems of quaternion differential equations  has much more complicated in comparing to usual  linear systems. Difficulties arise already in determining the quaternion determinant.

It's well-known that there are several approaches to the definition of  a determinant of matrices  with
noncommutative entries (which are also defined as noncommutative
determinants). The first approach is an axiomatic defining.
Let ${\rm M}\left( {n,\bf \mathcal{R}} \right)$ be
  the ring of $n\times n$ matrices with entries in a
ring $ {\bf \mathcal{R}}$.
\begin{definition}\cite{as,co} \label{defin:axiom} Let a
functional ${\rm d}:{\rm M}\left( {n,\bf \mathcal{R}} \right) \to \bf \mathcal{R}$
satisfy the following three axioms.
\begin{enumerate}
\item  ${\rm d}\left( {{\rm {\bf A}}} \right) = 0$ if and only if the matrix ${\rm {\bf A}}$ is singular.
\item  ${\rm d}\left( {{\rm {\bf A}} \cdot
{\rm {\bf B}}} \right) = {\rm d}\left( {{\rm {\bf A}}} \right)
\cdot {\rm d}\left( {{\rm {\bf B}}} \right)$ for $\forall {\rm
{\bf B}}\in {\rm M}\left( {n,\bf R} \right)$.
\item  If the matrix ${\rm {\bf A}}'$ is obtained
from ${\rm {\bf A}}$ by adding a left-multiple of a row to another
row or a right-multiple of a column to another column, then ${\rm
d}\left( {{\rm {\bf A}}}\right)'={\rm d}\left( {{\rm {\bf
A}}}\right)$.
\end{enumerate}
Then  ${\rm d}$ is called the determinant
of  ${\rm {\bf A}}\in {\rm M}\left( {n,\bf R}
\right)$.\end{definition}
But it turns out \cite{as}, if a determinant functional satisfies
Axioms 1, 2,
 3, then it takes on a value in a commutative subset of the ring.
 The famous examples of such determinant are the
 determinants of Diedonn\'{e} \cite{di} and Study \cite{stu}.

Another way of defining is constructive.
A noncommutative determinant is constructed by similar to the usual determinant as the alternative sum of $n!$ products of entries of a matrix but by specifying a certain ordering of coefficients in each term. The Caley determinant \cite{ca} has the such type but without success in the implementation of any of the axioms.  Moore \cite{mo} was the first who
 achieved the fulfillment of the
main Axiom 1 by such definition of  a noncommutative determinant.
This is done  not for all square matrices over a ring but rather
only Hermitian matrices.  Later, Dyson \cite{dy} gave some
natural generalizations, described the theory in more modern
terms, and represented  Moore's determinant  in terms of permutations as follows,
\[
{\rm{Mdet}}\, {\rm {\bf A}} = {\sum\limits_{\sigma \in S_{n}}
{{\left| {\sigma} \right|}{a_{n_{11} n_{12}}  \cdot \ldots \cdot
a_{n_{1l_{1}}  n_{11}}  \cdot }} {a_{n_{21} n_{22}}} \cdot \ldots
\cdot {a_{n_{rl_{1}}  n_{r1}}}}.
\]
 The disjoint
cycle representation of the permutation $\sigma \in S_{n}$ is
written in the normal form,
\[ \sigma = \left( {n_{11} \ldots
n_{1l_{1}} }  \right)\left( {n_{21} \ldots n_{2l_{2}} }
\right)\ldots \left( {n_{r1} \ldots n_{rl_{r}} }  \right),\]
where $n_{i1}<n_{im}$  for all $i=1,...,r$ and
$m>1$, and
$ n_{11}>n_{21}>...>n_{r1}$.
Dyson  has emphasized to the need for expansion of the definition of Moore's
determinant to arbitrary square matrices.
Chen has offered the
following decision of this problem in \cite{ch}. He has
defined the determinant of a square matrix ${\rm {\bf
A}}=(a_{ij}) \in {\rm M}\left( {n,\mathbb{H}} \right)$ over the
quaternion skew field $\mathbb{H}$ by putting,
\[
\begin{array}{c}
  \det {\rm {\bf A}} = {\sum\limits_{\sigma \in S_{n}}  {\varepsilon
\left( {\sigma}  \right)a_{n_{1} i_{2}}  \cdot a_{i_{2} i_{3}}
\ldots \cdot a_{i_{s} n_{1}}  \cdot} } \ldots \cdot a_{n_{r}
k_{2}}  \cdot \ldots \cdot a_{k_{l} n_{r}}, \\
   \sigma = \left( {n_{1} i_{2} \ldots i_{s}}  \right)\ldots \left(
{n_{r} k_{2} \ldots k_{l}}  \right),\\
  n_{1} > i_{2} ,i_{3} ,\ldots ,i_{s} ;\ldots ;n_{r} > k_{2} ,k_{3}
,\ldots ,k_{l}, \\
  n = n_{1} > n_{2} > \ldots > n_{r} \ge 1.
\end{array}
\]
Chen's determinant does not satisfy Axiom 1 and can not be expanded  by cofactors along an arbitrary row or
column with the exception of the $n$th row. Through by using this determinant, a determinantal representation of an inverse matrix over the quaternion skew field has been obtained.

Returning to the main topic, we note that in \cite{kou1,kou2,kou3} it has been used  the determinants of Caley and    Chen, and the double determinant is used that is actually the determinant of Study. In  \cite{kou4}, the same authors has  already used  the determinant of the complex adjoint matrix $\chi_{A}$ to a quaternion matrix ${\bf A}$. This indicates the complexity of the successful choice of a quaternion determinant.

In this paper we explore linear systems of quaternion differential equations applying the theory of column-row determinants introduced by the author in \cite{kyr1,kyr2}. By this theory, for a quaternion square matrix of order $n$ is defining $n$ column determinants and  $n$ row determinants that give the complete expansion of Moore's
determinant from Hermitian to arbitrary square matrices.
Currently, the theory of column-row determinants   is active  developing.
Within the framework of  column-row determinants, determinantal representations of various kind of generalized inverses and  (generalized inverses) solutions    of quaternion matrix equations recently have been derived as by the author (see, e.g.\cite{kyr3,kyr4,kyr5,kyr6,kyr7}) so by other researchers (see, e.g.\cite{kl,song1,song2,song3}).

Throughout the paper, we denote the real number field by ${\rm
{\mathbb{R}}}$, the set of all $m\times n$ matrices over the
quaternion algebra
\[{\rm {\mathbb{H}}}=\{a_{0}+a_{1}i+a_{2}j+a_{3}k\,
|\,i^{2}=j^{2}=k^{2}=ijk=-1,\, a_{0}, a_{1}, a_{2}, a_{3}\in{\rm
{\mathbb{R}}}\}\]
by ${\rm {\mathbb{H}}}^{m\times n}$, the identity matrix with the appropriate size by $ {\bf I}$.  Let ${\rm M}\left( {n,{\rm {\mathbb{H}}}} \right)$ be the
ring of $n\times n$ quaternion matrices. For ${\rm {\bf A}}
 \in {\rm {\mathbb{H}}}^{n\times m}$, the symbols ${\rm {\bf A}}^{ *}$ stands for the conjugate transpose (Hermitian adjoint) matrix
of ${\rm {\bf A}}$.
 The matrix ${\rm {\bf A}} = \left( {a_{ij}}  \right) \in {\rm
{\mathbb{H}}}^{n\times n}$ is Hermitian if ${\rm {\bf
A}}^{ *}  = {\rm {\bf A}}$.

The main goals of the paper are to give a systematic basic theory on  linear systems of quaternion differential equations, establish the algebraic structure of their general solutions, provide  algorithms for finding their solutions.

 The paper is organized as follows. We start with some
basic concepts and also some new results from the theories of quaternion-valued differential equations,  row and column determinants, the theory on quaternion vector spaces and  eigenvalues   of quaternion  matrices, and give an algorithm obtaining eigenvalues of quaternion normal matrices   in Section 2. In Section 3, we  consider first order right and left linear quaternion differential systems, establish the algebraic structure of their general solutions and, in particular, for systems  with constant coefficients, and give determinantal representations of solutions of  systems with constant coefficient matrices and sources vectors when coefficient matrices are invertible or singular. In Sections 2 an 3,  we show
 numerical examples to illustrate the main results.

\section{Preliminaries.}

\subsection{Background for quaternion-valued differential equations (QDE)}
Consider a quaternion-valued function of real variable, ${\bf  f}:{\mathbb{R}}\rightarrow {\mathbb{H}}$, ($t \in {\mathbb{R}}$ is a real variable), such that  ${\bf{ f}}(t)= f_{0}(t)+  f_{1}(t){\bf i}+  f_{2}(t){\bf j}+  f_{3}(t){\bf k}.$
The first
derivative of a quaternionic function ${\bf  f}(t)$ with respect to the real variable $t$ denote by,
\[{{\bf  f}'}(t):=\frac{{\rm d}{\bf  f}(t)}{{\rm d}t}=\frac{{\rm d}f_{0}(t)}{{\rm d}t}+\frac{{\rm d}f_{1}(t)}{{\rm d}t}{\bf i}+
\frac{{\rm d}f_{2}(t)}{{\rm d}t}{\bf j}+\frac{{\rm d}f_{3}(t)}{{\rm d}t}{\bf k}.\]
It is easy to prove the following proposition on properties of the derivative of quaternionic functions.
 \begin{proposition} If ${\bf q} :{\mathbb{R}}\rightarrow {\mathbb{H}}$ and ${\bf r} :{\mathbb{R}}\rightarrow {\mathbb{H}}$ are differentiable,
then $({\bf q}\pm {\bf r})(t)$, ${\bf qr}(t)$ and, for any integer $n \geq 1$, ${\bf q}^n$  are differentiable, and
\begin{gather*}
        ({\bf q}\pm {\bf r})'(t)= {\bf q}'(t)\pm {\bf r}'(t),\\
        ({\bf  qr})'(t)= {\bf q}'(t){\bf r}(t)+  {\bf q}(t){\bf r}'(t),\\
        \left[{\bf q}^n(t)\right]'=\sum_{j=0}^{n-1}{\bf q}^j(t){\bf q}'(t){\bf q}^{n-j}(t).
      \end{gather*}
\end{proposition}
We need the   exponential of $q\in{\mathbb{H}}$ that can be defined by putting,
\begin{equation}\label{def:exp} e^{q}=\sum_{n=0}^{\propto} \frac{q^n}{n!}.\end{equation}
From the definition, we evidently have the following properties.
\begin{proposition}\label{prop:e_der}\begin{enumerate}
                     \item If $q, r\in{\mathbb{H}}$ are such that $q r=rq$, then $e^{qr}=e^{q}+e^{r}$.
                     \item If ${\bf q} :{\mathbb{R}}\rightarrow {\mathbb{H}}$ is differentiable and ${\bf q}'(t){\bf q}(t)= {\bf q}(t){\bf q}'(t)$, then
\[ \left[e^{{\bf q}(t)}\right]'= \left[e^{{\bf q}(t)}\right]{\bf q}'(t)\]
                   \end{enumerate}
\end{proposition}
Consider  the  following  statement that we shall need below.
\begin{proposition}\label{prop:e}Let $q_{l}$, $l=1,2,3$, is one from the fundamental quaternion units, i.e. $q_{l}\in \{i,j,k\}$ for all $l=1,2,3$.
Then $$e^{q_{l}}q_{m}=\begin{cases}q_{m}e^{q_{l}}, & \mbox{if } m=l,\\
q_{m}e^{-q_{l}}, & \mbox{if } m\neq l,
\end{cases}
$$ where $m, l \in \{1, 2, 3\}$.
\end{proposition}
\begin{proof} If $m=l$, let $q_{m}=i$, then evidently by Definition \ref{def:exp}, $i$ and $e^{i}$ are the commuting elements. The same is true for the other two quaternion units, $j$ and $k$.

If $m\neq l$, we put $q_{m}=j$, $q_{l}=i$. Then
\begin{multline*}e^{i}j= \sum_{n} \frac{i^n}{n!}j=\begin{cases}\sum_{k} \frac{i^{2k}}{(2k)!}j,\\
\sum_{k} \frac{i^{2k+1}}{(2k+1)!}j,
\end{cases}=\begin{cases}\sum_{k} \frac{(-1)^{k}}{(2k)!}j,\\
\sum_{k} \frac{(-1)^{k}i}{(2k+1)!}j,
\end{cases}=\\\begin{cases}j\sum_{k} \frac{i^{2k}}{(2k)!},\\
j\sum_{k} \frac{-i^{2k+1}}{(2k+1)!},\end{cases}
=j\sum_{n} \frac{(-1)^ni^n}{n!}=j e^{-i}.
\end{multline*}
Similarly, we obtain for the all other  combinations of the quaternion units.
\end{proof}
If $f_{i}(t)$ for all $i = 0,\ldots,  3$ is integrable on $[a,b]\subset {\mathbb{R}}$, then ${\bf f}(t)$ is integrable as well, and
\[
\int_a^b {\bf f}(t){{\rm d}t}=\int_a^b f_0(t){{\rm d}t}+\int_a^b f_1(t){{\rm d}t}\,{\bf i}+\int_a^b f_2(t){{\rm d}t}\,{\bf j}+\int_a^b f_3(t){{\rm d}t}\,{\bf k}.
\]
In \cite{gup}, the linear quaternion differential equations,
 \begin{equation}\label{eq:left_dif1}
{\bf q}'(t)={\bf a}(t){\bf q}(t),
\end{equation}
and
 \begin{equation}\label{eq:right_dif1}
{\bf q}'(t)={\bf q}(t){\bf a}(t),
\end{equation} with the initial condition $q(t_{0})=q_{0}$
have been considered and  the following proposition has been derived.
\begin{proposition}\label{pr:lin_eq}
Let ${\bf q}(t)={\bf \Phi}_l(t)q_{0}$ and ${\bf q}(t)=q_{0}{\bf \Phi}_r(t)$ be solutions of (\ref{eq:left_dif1}) and (\ref{eq:right_dif1}), respectively. If
 \begin{equation}\label{eq:prop_com}{\bf a}(t)\int_{t_{0}}^t {\bf a}(\tau){{\rm d}\tau}=\int_{t_{0}}^t {\bf a}(\tau){{\rm d}\tau}\,{\bf a}(t),\end{equation} then
 \begin{equation}\label{eq:sol_l_eqs}
{\bf \Phi}_l(t)={\bf \Phi}_r(t)=e^{\int_{t_{0}}^t {\bf a}(\tau){{\rm d}t}}.
\end{equation}
If  ${\bf a}$ is constant, then $\int_{t_{0}}^t {\bf a}{{\rm d}\tau}={\bf a}\left(t-t_{0}\right)$, and ${\bf \Phi}_l(t)={\bf \Phi}_r(t)=e^{{\bf a}\left(t-t_{0}\right)}.$\end{proposition}
The similar result has been obtained in \cite{cam} as well. Moreover, in \cite{cam}, the following nonhomogeneous differential equations corresponding to (\ref{eq:left_dif1}) has been considered,
 \begin{equation}\label{eq:non_left_dif1}
{\bf q}'(t)={\bf a}(t){\bf q}(t)+{\bf f}(t),
\end{equation}
where ${\bf f} :[0, T]\rightarrow {\mathbb{H}}$ and ${\bf a} :[0, T]\rightarrow {\mathbb{H}}$. It has been shown if the condition (\ref{eq:prop_com}) is satisfied, then the solutions of (\ref{eq:non_left_dif1}) are given by
 \begin{equation*}{\bf q}(t)=e^{\int_{0}^t {\bf a}(\tau){{\rm d}\tau}}\left({\bf q}(0)+\int_{0}^t e^{\int_{0}^s \left( -{\bf a}(\tau)\right){{\rm d}\tau}}{\bf f}(s){{\rm d}s}\right),\;\;(t\in[0,T]).
\end{equation*}
In the special case when ${\bf a}$ is constant and ${\bf q}(0)=0$,  then the solutions of (\ref{eq:non_left_dif1}) are given by
 \begin{equation*}{\bf q}(t)=e^{{\bf a}t}\left(\int_{0}^t e^{ -{\bf a}s}{\bf f}(s){{\rm d}s}\right),\;\;(t\in[0,T]).
\end{equation*}
\begin{remark}The condition (\ref{eq:prop_com}) too narrow the set of solvable  linear quaternion differential equations. So, we propose to consider the two different chain rules for computing the derivative of the composition of two or more quaternion-valued functions, namely, left and right chain rules as follows.
\end{remark}
\begin{definition}\label{def:deriv}Let ${\bf q}_{i}$ are quaternion-valued differentiable functions for all $i=1,\ldots,n$ and ${\bf Q}(t)={\bf q}_{1}\left(\ldots\left({\bf q}_{n}\right)\ldots\right)(t)$. Then the left and right chain rules are defined by putting respectively  in the Lagrange notation,
\begin{gather*}
{\bf Q}'_{r}(t)={\bf q}'_{1..n}(t)={\bf q}'_{1}\left({\bf q}_{2..n}(t)\right)\cdot {\bf q}'_{2}\left({\bf q}_{3..n}(t)\right)\cdot\ldots\cdot{\bf q}'_{n}(t),\\
{\bf Q}'_{l}(t)={\bf q}'_{n..1}(t)={\bf q}'_{n}(t)\cdot\ldots\cdot{\bf q}'_{2}\left({\bf q}_{3..n}(t)\right)\cdot{\bf q}'_{1}\left({\bf q}_{2..n}(t)\right). \end{gather*}
\end{definition}
By Definition \ref{def:deriv}, considering the left linear equation (\ref{eq:left_dif1})
due only to the left chain rule and the right linear equation (\ref{eq:right_dif1})
 through  the right chain rule,
  we respectively obtain (\ref{eq:sol_l_eqs}) without a need to perform the condition (\ref{eq:prop_com}).
   We also have  change  the property of the derivative of the  quaternion exponent in Proposition \ref{prop:e_der} such that
\begin{equation*} \left[e^{{\bf q}(t)}\right]'_{l}= {\bf q}'(t)\left[e^{{\bf q}(t)}\right], \,\,\,
\left[e^{{\bf q}(t)}\right]'_{r}= \left[e^{{\bf q}(t)}\right]{\bf q}'(t),
 \end{equation*}
 without a need to commutativity ${\bf q}(t)$ and ${\bf q}'(t)$.

 \subsection{Elements of the theory of  column and row determinants}
Suppose $S_{n}$ is the symmetric group on the set $I_{n}=\{1,\ldots,n\}$. For  ${\rm {\bf A}}=(a_{ij}) \in {\rm
M}\left( {n,{\mathbb{H}}} \right)$ we define $n$ row determinants and $n$ column determinants as follows.
\begin{definition}\cite{kyr1}
 The $i$th row determinant of ${\rm {\bf A}}=(a_{ij}) \in {\rm
M}\left( {n,{\mathbb{H}}} \right)$ is defined  for all $i = {1,\ldots,n} $
by putting
 \begin{gather*}{\rm{rdet}}_{ i} {\rm {\bf A}} =
{\sum\limits_{\sigma \in S_{n}} {\left( { - 1} \right)^{n - r}({a_{i{\kern
1pt} i_{k_{1}}} } {a_{i_{k_{1}}   i_{k_{1} + 1}}} \ldots } } {a_{i_{k_{1}
+ l_{1}}
 i}})  \ldots  ({a_{i_{k_{r}}  i_{k_{r} + 1}}}
\ldots  {a_{i_{k_{r} + l_{r}}  i_{k_{r}} }}),\\
\sigma = \left(
{i\,i_{k_{1}}  i_{k_{1} + 1} \ldots i_{k_{1} + l_{1}} } \right)\left(
{i_{k_{2}}  i_{k_{2} + 1} \ldots i_{k_{2} + l_{2}} } \right)\ldots \left(
{i_{k_{r}}  i_{k_{r} + 1} \ldots i_{k_{r} + l_{r}} } \right),\end{gather*}
with
conditions $i_{k_{2}} < i_{k_{3}}  < \ldots < i_{k_{r}}$ and $i_{k_{t}}  <
i_{k_{t} + s} $ for $t = {2,\ldots,r} $ and $s ={1,\ldots,l_{t}} $.
\end{definition}
\begin{definition}\cite{kyr1}
The $j$th column determinant
 of ${\rm {\bf
A}}=(a_{ij}) \in {\rm M}\left( {n,{\mathbb{H}}} \right)$ is defined for
all $j ={1,\ldots,n} $ by putting
 \begin{gather*}{\rm{cdet}} _{{j}}\, {\rm {\bf A}} =
\sum\limits_{\tau \in S_{n}} \left( { - 1} \right)^{n - r}(a_{j_{k_{r}}
j_{k_{r} + l_{r}} } \ldots a_{j_{k_{r} + 1} i_{k_{r}} } ) \ldots  (a_{j\,
j_{k_{1} + l_{1}} }  \ldots  a_{ j_{k_{1} + 1} j_{k_{1}} }a_{j_{k_{1}}
j}),\\
\tau =
\left( {j_{k_{r} + l_{r}}  \ldots j_{k_{r} + 1} j_{k_{r}} } \right)\ldots
\left( {j_{k_{2} + l_{2}}  \ldots j_{k_{2} + 1} j_{k_{2}} } \right){\kern
1pt} \left( {j_{k_{1} + l_{1}}  \ldots j_{k_{1} + 1} j_{k_{1} } j}
\right), \end{gather*}
\noindent with conditions, $j_{k_{2}}  < j_{k_{3}}  < \ldots <
j_{k_{r}} $ and $j_{k_{t}}  < j_{k_{t} + s} $ for  $t = {2,\ldots,r} $
and $s = {1,\ldots,l_{t}}  $.
\end{definition}

Suppose ${\rm {\bf A}}_{}^{i{\kern 1pt} j} $ denotes the submatrix of
${\rm {\bf A}}$ obtained by deleting both the $i$th row and the $j$th
column. Let ${\rm {\bf a}}_{.j} $ be the $j$th column and ${\rm {\bf
a}}_{i.} $ be the $i$th row of ${\rm {\bf A}}$. Suppose ${\rm {\bf
A}}_{.j} \left( {{\rm {\bf b}}} \right)$ denotes the matrix obtained from
${\rm {\bf A}}$ by replacing its $j$th column with the column ${\rm {\bf
b}}$, and ${\rm {\bf A}}_{i.} \left( {{\rm {\bf b}}} \right)$ denotes the
matrix obtained from ${\rm {\bf A}}$ by replacing its $i$th row with the
row ${\rm {\bf b}}$.

The following theorem has a key value in the theory of the column and row
determinants.
\begin{theorem} \cite{kyr1}\label{theorem:
determinant of hermitian matrix} If ${\rm {\bf A}} = \left( {a_{ij}}
\right) \in {\rm M}\left( {n,{\rm {\mathbb{H}}}} \right)$ is Hermitian,
then ${\rm{rdet}} _{1} {\rm {\bf A}} = \cdots = {\rm{rdet}} _{n} {\rm {\bf
A}} = {\rm{cdet}} _{1} {\rm {\bf A}} = \cdots = {\rm{cdet}} _{n} {\rm {\bf
A}} \in {\rm {\mathbb{R}}}.$
\end{theorem}
Due to Theorem \ref{theorem:
determinant of hermitian matrix}, we can define the
determinant of a  Hermitian matrix ${\rm {\bf A}}\in {\rm M}\left( {n,{\rm
{\mathbb{H}}}} \right)$. By definition, we put
$\det {\rm {\bf A}}: = {\rm{rdet}}_{{i}}\,
{\rm {\bf A}} = {\rm{cdet}} _{{i}}\, {\rm {\bf A}}$
 for all $i ={1,\ldots,n}$.
The determinant of a Hermitian matrix has properties similar to a usual determinant. They are completely explored
in
 \cite{kyr1, kyr2}
 by its row and
column determinants and can be summarized by the following
theorems.
\begin{theorem}\label{theorem:row_combin} If the $i$th row of
a Hermitian matrix ${\rm {\bf A}}\in {\rm M}\left( {n,{\rm
{\mathbb{H}}}} \right)$ is replaced with a left linear combination
of its other rows, i.e. ${\rm {\bf a}}_{i.} = c_{1} {\rm {\bf
a}}_{i_{1} .} + \ldots + c_{k}  {\rm {\bf a}}_{i_{k} .}$, where $
c_{l} \in {{\rm {\mathbb{H}}}}$ for all $ l = {1,\ldots, k}$ and
$\{i,i_{l}\}\subset I_{n} $, then
\[
 {\rm{rdet}}_{i}\, {\rm {\bf A}}_{i \, .} \left(
{c_{1} {\rm {\bf a}}_{i_{1} .} + \ldots + c_{k} {\rm {\bf
a}}_{i_{k} .}}  \right) = {\rm{cdet}} _{i}\, {\rm {\bf A}}_{i\, .}
\left( {c_{1}
 {\rm {\bf a}}_{i_{1} .} + \ldots + c_{k} {\rm {\bf
a}}_{i_{k} .}}  \right) = 0.
\]
\end{theorem}
\begin{theorem}\label{theorem:colum_combin} If the $j$th column of
 a Hermitian matrix ${\rm {\bf A}}\in
{\rm M}\left( {n,{\rm {\mathbb{H}}}} \right)$   is replaced with a
right linear combination of its other columns, i.e. ${\rm {\bf
a}}_{.j} = {\rm {\bf a}}_{.j_{1}}   c_{1} + \ldots + {\rm {\bf
a}}_{.j_{k}} c_{k} $, where $c_{l} \in{{\rm {\mathbb{H}}}}$ for
all $ l = {1,\ldots,k}$ and $\{j,j_{l}\}\subset J_{n}$, then
 \[{\rm{cdet}} _{j}\, {\rm {\bf A}}_{.j}
\left( {{\rm {\bf a}}_{.j_{1}} c_{1} + \ldots + {\rm {\bf
a}}_{.j_{k}}c_{k}} \right) ={\rm{rdet}} _{j} \,{\rm {\bf A}}_{.j}
\left( {{\rm {\bf a}}_{.j_{1}}  c_{1} + \ldots + {\rm {\bf
a}}_{.j_{k}}  c_{k}} \right) = 0.
\]
\end{theorem}
  The determinant of a Hermitian matrix also has a property of expansion along arbitrary rows and columns using  row and column determinants of submatrices. So, we were able to get determinantal representations of an inverse  as follows.
\begin{theorem}\cite{kyr1} \label{th:inver_her} If ${\rm {\bf A}}\in {\rm M}\left( {n,{\rm
{\mathbb{H}}}} \right)$ is Hermitian and
$\det {\rm {\bf A}} \ne 0,$
then there exist its unique  right  $(R{\rm {\bf A}})^{ - 1}$
and  unique left inverse  $(L{\rm {\bf A}})^{ - 1}$, where $\left( {R{\rm {\bf A}}} \right)^{ - 1} = \left( {L{\rm {\bf A}}}
\right)^{ - 1} = :{\rm {\bf A}}^{ - 1}$, and they possess the following determinantal representations,
\begin{multline*}\label{eq:inver_her_R}
  \left( {R{\rm {\bf A}}} \right)^{ - 1} = {\frac{{1}}{{\det {\rm
{\bf A}}}}}
\begin{bmatrix}
  R_{11} & R_{21} & \cdots & R_{n1}\\
  R_{12} & R_{22} & \cdots & R_{n2}\\
  \cdots & \cdots & \cdots& \cdots\\
  R_{1n} & R_{2n} & \cdots & R_{nn}
\end{bmatrix},
\\
  \left( {L{\rm {\bf A}}} \right)^{ - 1} = {\frac{{1}}{{\det {\rm
{\bf A}}}}}
\begin{bmatrix}
  L_{11} & L_{21} & \cdots & L_{n1} \\
  L_{12} & L_{22} & \cdots & L_{n2} \\
  \cdots & \cdots & \cdots & \cdots \\
  L_{1n} & L_{2n} & \cdots & L_{nn}
\end{bmatrix},
\end{multline*}
where $ \det {\rm {\bf A}} ={\sum\limits_{j = 1}^{n} {{a_{i{\kern
1pt} j} \cdot R_{i{\kern 1pt} j} } }}= {{\sum\limits_{i = 1}^{n}
{L_{i{\kern 1pt} j} \cdot a_{i{\kern 1pt} j}} }}$,
\begin{gather*}
 R_{i{\kern 1pt} j} = {\left\{ {{\begin{array}{*{20}c}
  - {\rm{rdet}}_{{j}}\, {\rm {\bf A}}_{{.{\kern 1pt} j}}^{{i{\kern 1pt} i}} \left( {{\rm
{\bf a}}_{{.{\kern 1pt} {\kern 1pt} i}}}  \right),& {i \ne j},
\hfill \\
 {\rm{rdet}} _{{k}}\, {\rm {\bf A}}^{{i{\kern 1pt} i}},&{i = j},
\hfill \\
\end{array}} } \right.}\,\,\,\,
L_{i{\kern 1pt} j} = {\left\{ {\begin{array}{*{20}c}
 -{\rm{cdet}} _{i}\, {\rm {\bf A}}_{i{\kern 1pt} .}^{j{\kern 1pt}j} \left( {{\rm {\bf a}}_{j{\kern 1pt}. } }\right),& {i \ne
j},\\
 {\rm{cdet}} _{k}\, {\rm {\bf A}}^{j\, j},& {i = j},
\\
\end{array} }\right.}
\end{gather*}
 and ${\rm {\bf A}}_{.{\kern 1pt} j}^{i{\kern 1pt} i} \left(
{{\rm {\bf a}}_{.{\kern 1pt} {\kern 1pt} i}}  \right)$ is obtained from
${\rm {\bf A}}$   by both replacing the $j$th column with the $i$th column and deleting  the $i$th row and column, ${\rm {\bf A}}_{i{\kern 1pt} .}^{jj} \left( {{\rm {\bf
a}}_{j{\kern 1pt} .} } \right)$ is obtained
 by replacing the $i$th row with the $j$th row, and then by
deleting both the $j$th row and  column, respectively, $I_{n} =  {\left\{ {1,\ldots ,n} \right\}}$, $k = \min {\left\{ {I_{n}}
\right.} \setminus {\left. {\{i\}} \right\}}$
 for all $ i,j =
{1,\ldots,n}$.
\end{theorem}
\begin{theorem}\cite{kyr1}\label{theorem:crit_depen} The right linearly dependence of
columns of ${\rm {\bf A}} \in {{\rm {\mathbb{H}}}}^{m\times n}$ or
the left linearly dependence of rows of ${\rm {\bf A}}^{ *} $ is
the necessary and sufficient condition for $\det {\rm {\bf A}}^{
*} {\rm {\bf A}} = 0.$
\end{theorem}
\begin{theorem}\label{theorem:equal_lef_righ_her} If ${\rm {\bf A}} \in {\rm M}\left(
{n,{{\rm {\mathbb{H}}}}} \right)$, then $\det {\rm {\bf A}}{\rm
{\bf A}}^{ *} = \det {\rm {\bf A}}^{ * }{\rm {\bf A}}$.
\end{theorem}
\begin{definition}For ${\rm {\bf A}} \in {\rm M}\left(
{n,{{\rm {\mathbb{H}}}}} \right)$, the double determinant of  ${\bf A}$ is defined by putting,
${\rm ddet}\,{\bf A}:= \det {\rm {\bf A}}{\rm
{\bf A}}^{ *} = \det {\rm {\bf A}}^{ * }{\rm {\bf A}}.$
\end{definition}
\begin{theorem} \label{theorem:deter_inver}\cite{kyr2} The necessary and sufficient condition of invertibility
of  ${\rm {\bf A}} \in {\rm M}(n,{{\rm {\mathbb{H}}}})$ is
${\rm{ddet}} {\rm {\bf A}} \ne 0$. Then there exists ${\rm {\bf
A}}^{ - 1} = \left( {L{\rm {\bf A}}} \right)^{ - 1} = \left(
{R{\rm {\bf A}}} \right)^{ - 1}$, where
\begin{equation*}
 \left( {L{\rm {\bf A}}} \right)^{ - 1}
=\left( {{\rm {\bf A}}^{ *}{\rm {\bf A}} } \right)^{ - 1}{\rm {\bf
A}}^{ *} ={\frac{{1}}{{{\rm{ddet}}{ \rm{\bf A}} }}}
\begin{bmatrix}
  {\mathbb{L}} _{11} & {\mathbb{L}} _{21}& \ldots & {\mathbb{L}} _{n1} \\
  {\mathbb{L}} _{12} & {\mathbb{L}} _{22} & \ldots & {\mathbb{L}} _{n2} \\
  \ldots & \ldots & \ldots & \ldots \\
 {\mathbb{L}} _{1n} & {\mathbb{L}} _{2n} & \ldots & {\mathbb{L}} _{nn}
\end{bmatrix}
\end{equation*}
\begin{equation*} \left( {R{\rm {\bf A}}} \right)^{ - 1} = {\rm {\bf
A}}^{ *} \left( {{\rm {\bf A}}{\rm {\bf A}}^{ *} } \right)^{ - 1}
= {\frac{{1}}{{{\rm{ddet}}{ \rm{\bf A}}^{ *} }}}
\begin{bmatrix}
 {\mathbb{R}} _{11} & {\mathbb{R}} _{21} &\ldots & {\mathbb{R}} _{n1} \\
 {\mathbb{R}} _{12} & {\mathbb{R}} _{22} &\ldots & {\mathbb{R}} _{n2}  \\
 \ldots  & \ldots & \ldots & \ldots \\
 {\mathbb{R}} _{1n} & {\mathbb{R}} _{2n} &\ldots & {\mathbb{R}} _{nn}
\end{bmatrix}
\end{equation*}
and \[{\mathbb{L}} _{ij} = {\rm{cdet}} _{j} ({\rm {\bf
A}}^{\ast}{\rm {\bf A}})_{.j} \left( {{\rm {\bf a}}_{.{\kern 1pt}
i}^{ *} } \right), \,\,\,{\mathbb{R}} _{\,{\kern 1pt} ij} =
{\rm{rdet}}_{i} ({\rm {\bf A}}{\rm {\bf A}}^{\ast})_{i.} \left(
{{\rm {\bf a}}_{j.}^{ *} }  \right),\] for all $i,j =
{1,\ldots,n}.$
\end{theorem}
Moreover, the following  criterion of invertibility of a quaternion matrix
can be obtained.
\begin{theorem}\cite{kyr2}\label{theorem:inver_equiv} If ${\rm {\bf A}}\in {\rm M}\left( {n,{\rm
{\mathbb{H}}}} \right)$, then
 the following statements are equivalent.
\begin{itemize}
  \item[i)] ${\rm {\bf A}}$ is invertible, i.e. ${\rm {\bf A}}\in
GL\left( {n,{\mathbb{H}}} \right);$
  \item[ii)]  rows of ${\rm {\bf A}}$ are left-linearly independent;
  \item[iii)]  columns of ${\rm {\bf A}}$ are right-linearly
independent;
  \item[iv)] ${\rm ddet}\, {\rm {\bf A}}\ne 0$.
\end{itemize}
\end{theorem}
Due to Theorems \ref{th:inver_her}  and \ref{theorem:deter_inver}, we evidently obtain the following Cramer rules.
\begin{theorem}\cite{kyr2} \label{theorem:right_system}Let
\begin{equation}
\label{eq:right_syst} {\rm {\bf A}} \cdot {\rm {\bf x}} = {\rm
{\bf b}}
\end{equation}
be a right system of  linear equations  with a matrix of
coefficients ${\bf A}\in {\rm M}(n, {\mathbb{H}})$, a
column of constants ${\bf b} = \left( {b_{1} ,\ldots ,b_{n}
} \right)\in  {\mathbb{H}}^{n\times 1}$, and a
column of unknowns $ {\bf x} = \left( {x_{1} ,\ldots ,x_{n}}
\right)^{T}$.
\begin{itemize}
  \item[(i)]If ${\bf A}$ is Hermitian and ${\det}{\bf A} \ne 0$, then the
solution to the linear system (\ref{eq:right_syst}) is given by
components
\begin{equation*} x_{j} = {\frac{{{\rm{cdet}} _{j}
 {\bf A}_{.j} \left( {{\rm {\bf b}}}
\right)}}{{{\det}  {\bf A}}}}.
\end{equation*}
  \item[(ii)] If ${\bf A}$ is an arbitrary matrix and ${\rm{ddet}}{ \rm{\bf A}} \ne 0$, then the
solution to the linear system (\ref{eq:right_syst}) is given by
components
\begin{equation*} x_{j} = {\frac{{{\rm{cdet}} _{j}
({\rm {\bf A}}^{ *} {\rm {\bf A}})_{.j} \left( {{\rm {\bf f}}}
\right)}}{{{\rm{ddet}} {\rm {\bf A}}}}}, \,\,\, j =
 {1,\ldots,n},
\end{equation*}
where $ {\bf f} =  {\bf A}^{ *}  {\bf b}.$
\end{itemize}
\end{theorem}
\begin{theorem}\cite{kyr2}\label{theorem:left_system} Let
\begin{equation}\label{eq:left_syst}
 {\bf x} \cdot  {\bf A} =  {\bf b}
\end{equation}
be a left  system of  linear equations
 with a
matrix of coefficients $ {\bf A}\in {\rm M}(n,
{\mathbb{H}})$, a row of constants $ {\bf b} = \left(
{b_{1} ,\ldots ,b_{n} } \right)\in
{\mathbb{H}}^{1\times n}$, and a row of unknowns $ {\bf x}
= \left( {x_{1} ,\ldots ,x_{n}} \right)$.
\begin{itemize}
\item[(i)]If ${\bf A}$ is Hermitian and ${\det}{\bf A} \ne 0$, then the
solution to (\ref{eq:left_syst}) is given by
components
\begin{equation*} x_{i} = {\frac{{{\rm{rdet}} _{i}
{\bf A}_{i.} \left( {{\rm {\bf b}}}
\right)}}{{{\det}  {\bf A}}}}.
\end{equation*}
  \item[(ii)]If ${\rm{ddet}}{ \rm{\bf
A}} \ne 0$, then the solution to the linear system
(\ref{eq:left_syst}) is given by components
\begin{equation*}
  x_{i} = {\frac{{{\rm{rdet}}_{i} ( {\bf A} {\bf A}^{ *})_{i.} \left(
{\bf z} \right)}}{{{\rm{ddet}}{\bf A}}}}, \quad  i =
 {1,\ldots,n},
\end{equation*}
where $  {\bf z} =  {\bf b} {\bf A}^{ *}.$
\end{itemize}
\end{theorem}

For determinantal representations of quaternion generalized inverses,
we are needed by the following notations. Let $\alpha : = \left\{
{\alpha _{1} ,\ldots ,\alpha _{k}} \right\} \subseteq {\left\{
{1,\ldots ,m} \right\}}$ and $\beta : = \left\{ {\beta _{1}
,\ldots ,\beta _{k}} \right\} \subseteq {\left\{ {1,\ldots ,n}
\right\}}$ be subsets of the order $1 \le k \le \min {\left\{
{m,n} \right\}}$. By ${\rm {\bf A}}_{\beta} ^{\alpha} $ denote the
submatrix of ${\rm {\bf A}}$ determined by the rows indexed by
$\alpha$ and the columns indexed by $\beta$. Then ${\rm {\bf
A}}{\kern 1pt}_{\alpha} ^{\alpha}$ denotes the principal submatrix
determined by the rows and columns indexed by $\alpha$.
 If ${\rm {\bf A}} \in {\rm
M}\left( {n,{\rm {\mathbb{H}}}} \right)$ is Hermitian, then by
${\left| {{\rm {\bf A}}_{\alpha} ^{\alpha} } \right|}$ denote the
corresponding principal minor of $\det {\rm {\bf A}}$.
 For $1 \leq k\leq n$, the collection of strictly
increasing sequences of $k$ integers chosen from $\left\{
{1,\ldots ,n} \right\}$ is denoted by $\textsl{L}_{ k,
n}: = {\left\{ {\,\alpha :\alpha = \left( {\alpha _{1} ,\ldots
,\alpha _{k}} \right),\,{\kern 1pt} 1 \le \alpha _{1} \le \ldots
\le \alpha _{k} \le n} \right\}}$.  For fixed $i \in \alpha $ and $j \in
\beta $, let $I_{r,\,m} {\left\{ {i} \right\}}: = {\left\{
{\,\alpha :\alpha \in L_{r,m} ,i \in \alpha}  \right\}}{\rm ,}
\quad J_{r,\,n} {\left\{ {j} \right\}}: = {\left\{ {\,\beta :\beta
\in L_{r,n} ,j \in \beta}  \right\}}$.

The Drazin inverse for an arbitrary square matrix ${\bf A}$ over $\mathbb{H}$ is defined to be the unique matrix ${\bf X}$ that satisfying the following equations \cite{dr},
\begin{gather*}{\bf X} {\rm {\bf
A}}{\bf X}  = {\bf X};\,
                                   {\rm {\bf A}}{\bf X}  = {\rm
{\bf X}{\bf A}}; \,
                                   {\rm {\bf A}}^{k+1}{\rm
{\bf X}}={\rm {\bf
         A}}^{k}. \end{gather*}
where $k = Ind\,{\bf A}$ is the smallest positive number such that $\rank {\bf A}^{k+1}=\rank {\bf A}^{k}$. It is denoted by ${\rm {\bf X}}={\bf A}^{D}$.

Denote by ${\rm {\bf a}}_{.j}^{(m)} $ and ${\rm {\bf
a}}_{i.}^{(m)} $ the $j$th column  and the $i$th row of  ${\rm
{\bf A}}^{m} $, and by $\hat{{\rm {\bf a}}}_{.s}$ and $\check{{{\rm {\bf a}}}}_{t.}$ the $s$th column of $
({\bf A}^{ 2k+1})^{*} {\bf A}^{k}=:\hat{{\rm {\bf A}}}=
(\hat{a}_{ij})\in {\mathbb{H}}^{n\times n}$
and  the $t$th row of $
{\bf A}^{k}({\bf A}^{ 2k+1})^{*} =:\check{{\rm {\bf A}}}=
(\check{a}_{ij})\in {\mathbb{H}}^{n\times n}$, respectively,  for all $s,t=
{1,\ldots.n}$.
\begin{theorem} \cite{kyr4}\label{theor:det_rep_draz} If ${\rm {\bf A}} \in {\rm M}\left( {n, {\mathbb{H}}}\right)$  with
$ Ind{\kern 1pt} {\rm {\bf A}}=k$ and $\rank{\rm {\bf A}}^{k+1} =
\rank{\rm {\bf A}}^{k} = r$, then the Drazin inverse  ${\rm {\bf A}}^{ D} $ possess the
determinantal representations,
\begin{equation}
\label{eq:cdet_draz} a_{ij} ^{D}=
 {\frac{{ \sum\limits_{t = 1}^{n} {a}_{it}^{(k)}   {\sum\limits_{\beta \in J_{r,\,n} {\left\{ {t}
\right\}}} {{\rm{cdet}} _{t} \left( {\left({\bf A}^{ 2k+1} \right)^{*}\left({\bf A}^{ 2k+1} \right)_{. \,t} \left( \hat{{\rm {\bf a}}}_{.\,j}
\right)} \right){\kern 1pt}  _{\beta} ^{\beta} } }
}}{{{\sum\limits_{\beta \in J_{r,\,n}} {{\left| {\left({\bf A}^{ 2k+1} \right)^{*}\left({\bf A}^{ 2k+1} \right){\kern 1pt} _{\beta} ^{\beta}
}  \right|}}} }}}
\end{equation}
and
\begin{equation}
\label{eq:rdet_draz} a_{ij} ^{D}=
{\frac{\sum\limits_{s = 1}^{n}\left({{\sum\limits_{\alpha \in I_{r,\,n} {\left\{ {s}
\right\}}} {{\rm{rdet}} _{s} \left( {\left( { {\bf A}^{ 2k+1} \left({\bf A}^{ 2k+1} \right)^{*}
} \right)_{\,. s} (\check{{\rm {\bf a}}}_{i\,.})} \right) {\kern 1pt} _{\alpha} ^{\alpha} } }
}\right){a}_{sj}^{(k)}}{{{\sum\limits_{\alpha \in I_{r,\,n}} {{\left| {\left( { {\bf A}^{ 2k+1} \left({\bf A}^{ 2k+1} \right)^{*}
} \right){\kern 1pt} _{\alpha} ^{\alpha}
}  \right|}}} }}}
\end{equation}
\end{theorem}
In the special case, when ${\rm {\bf A}} \in {\rm M}\left( {n, {\mathbb{H}}}\right)$ is Hermitian, we can obtain  simpler determinantal
representations of the Drazin inverse.
\begin{theorem} \cite{kyr4}\label{theor:det_rep_draz_her}
If ${\rm {\bf A}} \in {\rm M}\left( {n, {\mathbb{H}}}\right)$ is Hermitian with
$ Ind{\kern 1pt} {\rm {\bf A}}=k$ and $\rank{\rm {\bf A}}^{k+1} =
\rank{\rm {\bf A}}^{k} = r$, then the Drazin inverse ${\rm {\bf
A}}^{D} = \left( {a_{ij}^{D} } \right) \in {\rm
{\mathbb{H}}}_{}^{n\times n} $ possess the following determinantal
representations,
\begin{equation}
\label{eq:dr_rep_cdet} a_{ij}^{D}  = {\frac{{{\sum\limits_{\beta
\in J_{r,\,n} {\left\{ {i} \right\}}} {{\rm{cdet}} _{i} \left(
{\left( {{\rm {\bf A}}^{k+1}} \right)_{\,. \,i} \left( {{\rm {\bf
a}}_{.j}^{ k} }  \right)} \right){\kern 1pt} {\kern 1pt} _{\beta}
^{\beta} } } }}{{{\sum\limits_{\beta \in J_{r,\,\,n}} {{\left|
{\left( {{\rm {\bf A}}^{k+1}} \right){\kern 1pt} _{\beta} ^{\beta}
}  \right|}}} }}},
\end{equation}
or
\begin{equation}
\label{eq:dr_rep_rdet} a_{ij}^{D}  = {\frac{{{\sum\limits_{\alpha
\in I_{r,n} {\left\{ {j} \right\}}} {{\rm{rdet}} _{j} \left(
{({\rm {\bf A}}^{ k+1} )_{j\,.\,} ({\rm {\bf a}}_{i.\,}^{ (k)} )}
\right)\,_{\alpha} ^{\alpha} } }}}{{{\sum\limits_{\alpha \in
I_{r,\,n}}  {{\left| {\left( {{\rm {\bf A}}^{k+1} } \right){\kern
1pt}  _{\alpha} ^{\alpha} } \right|}}} }}}.
\end{equation}
\end{theorem}

\subsection{Some provisions of quaternion matrices and  pre-Hilbert spaces.}

Due to quaternion-scalar multiplying on the right, quaternionic column-vectors  form a right  vector $\mathbb{H}$-space, and, by quaternion-scalar multiplying on the left, quaternionic row-vectors  form a left  vector $\mathbb{H}$-space.
Moreover, we define   right  and left   quaternionic vector spaces, denoted by $\mathcal{H}_{r}$ and $\mathcal{H}_{l}$, respectively,  with corresponding $\mathbb{H}$-valued inner products $\langle\cdot,\cdot\rangle$ which satisfy, for every $\alpha, \beta \in\mathbb{H}$, and $\mathbf{x}, \mathbf{y}, \mathbf{z} \in \mathcal{H}_{r} (\mathcal{H}_{l})$, the  relations:
\begin{enumerate}
  \item $\langle \mathbf{x},\mathbf{y}\rangle=\overline{\langle \mathbf{y},\mathbf{x}\rangle}$;
  \item $\langle \mathbf{x},\mathbf{x}\rangle\geq 0\in\mathbb{R}$ and $\|\mathbf{x}\|^{2}:=\langle \mathbf{x},\mathbf{x}\rangle=0\,\, \Leftrightarrow \,\,\mathbf{x}=\mathbf{0}$;
  \item \begin{description}
          \item[] $\langle \mathbf{x}\alpha+\mathbf{y}\beta,\mathbf{z} \rangle=\langle \mathbf{x},\mathbf{z}\rangle\alpha+\langle \mathbf{y},\mathbf{z}\rangle\beta$ when $\mathbf{x}, \mathbf{y}, \mathbf{z} \in \mathcal{H}_{r}$;
          \item []$\langle \alpha\mathbf{x}+\beta\mathbf{y},\mathbf{z} \rangle=\alpha\langle \mathbf{x},\mathbf{z}\rangle+\beta\langle \mathbf{y},\mathbf{z}\rangle$ when $\mathbf{x}, \mathbf{y}, \mathbf{z} \in \mathcal{H}_{l}$;
        \end{description}
  \item \begin{description}
          \item[] $\langle \mathbf{x},\mathbf{y}\alpha+\mathbf{z}\beta \rangle=\overline{\alpha}\langle \mathbf{x},\mathbf{y}\rangle+\overline{\beta}\langle \mathbf{x},\mathbf{z}\rangle$ when $\mathbf{x}, \mathbf{y}, \mathbf{z} \in \mathcal{H}_{r}$;
          \item[] $\langle \mathbf{x}, \alpha\mathbf{y}+\beta\mathbf{z} \rangle=\langle \mathbf{x},\mathbf{y}\rangle\overline{\alpha}+\langle \mathbf{y},\mathbf{z}\rangle\overline{\beta}$ when $\mathbf{x}, \mathbf{y}, \mathbf{z} \in \mathcal{H}_{l}$.
        \end{description}
        It can be achieved by putting $\langle \mathbf{x},\mathbf{y}\rangle_{r}=\overline{y}_1x_{1}+\cdots+\overline{y}_{n}x_{n}$ for $\mathbf{x}=\left(x_{i}\right)_{i=1}^{n}, \mathbf{y}=\left(y_{i}\right)_{i=1}^{n}\in \mathcal{H}_{r}$, and $\langle \mathbf{x},\mathbf{y}\rangle_{l}=x_{1}\overline{y}_1+\cdots+x_{n}\overline{y}_{n}$ for $\mathbf{x}, \mathbf{y}\in \mathcal{H}_{l}$.
\end{enumerate}
The right   vector spaces $\mathcal{H}_{r}$ possess  the Gram-Schmidt process which takes a nonorthogonal set of linearly independent vectors $S = \{{\bf v}_1, ..., {\bf v}_k\}$ for $k \leq n$ and constructs an orthogonal (or orthonormal) basis  $S^{'} = \{{\bf u}_1, ..., {\bf u}_k\}$ that spans the same $k$-dimensional subspace of $\mathcal{H}_{r}$ as $S$. To $\mathcal{H}_{r}$, the following projection operator  is defined by
$$
{\rm proj}^{r}_{\bf u}({\bf v}):= {\bf u}\frac{\langle \mathbf{u},\mathbf{v}\rangle_{r}}{\langle \mathbf{u},\mathbf{u}\rangle_{r}},
$$
which orthogonally projects the vector ${\bf v}$  onto the line spanned by the vector ${\bf u}$. Then, the Gram-Schmidt process works as follows,
$$
\begin{array}{ll}
 {\bf u}_1 ={\bf v}_1, & {\bf e}_1=\frac{{\bf u}_1}{\|{\bf u}_1\|}, \\
  {\bf u}_k ={\bf v}_k - \sum_{j=1}^{k-1}{\rm proj}^{r}_{{\bf u}_{j}}({\bf v_{k}}),& {\bf e}_k=\frac{{\bf u}_k}{\|{\bf u}_k\|}.
\end{array}
$$
The sequence ${\bf u}_1, ..., {\bf u}_k$ is the required system of orthogonal vectors, and the normalized vectors ${\bf e}_1, ..., {\bf e}_k$ form an orthonormal set.

The Gram-Schmidt process for the left   vector spaces $\mathcal{H}_{l}$ can be realize by the same algorithm but with the projection operator
$$
{\rm proj}^{l}_{\bf u}({\bf v}):= \frac{\langle \mathbf{u},\mathbf{v}\rangle_{l}}{\langle \mathbf{u},\mathbf{u}\rangle_{l}}{\bf u}.
$$
\begin{definition}
Suppose ${\rm {\bf U}} \in {\rm M}\left( {n,{\rm {\mathbb{H}}}}
\right)$ and ${\rm {\bf U}}^{ *} {\rm {\bf U}} = {\rm {\bf U}}{\rm
{\bf U}}^{ *}  = {\rm {\bf I}}$, then the matrix ${\rm {\bf U}}$
is called unitary.
\end{definition}
Clear, that columns of ${\bf U}$ form a system of normalized vectors in  $\mathcal{H}_{r}$, rows of ${\bf U}^{*}$ is a system of normalized vectors in  $\mathcal{H}_{l}$.

The  vector norms $\|\mathbf{x}\|_{r}=\sqrt{\langle \mathbf{x},\mathbf{x}\rangle}_{r}$ and $\|\mathbf{x}\|_{l}=\sqrt{\langle \mathbf{x},\mathbf{x}\rangle}_{l}$ on  $\mathcal{H}_{r}$ and  $\mathcal{H}_{l}$, respectively, define  the corresponding induced matrix norms  on the space ${\mathbb{H}}^{n\times n}$ of all  $n\times n$ matrices as follows:
\begin{gather*}
\|{\bf A}\|_{r}=\sup\{\|{\bf A}\mathbf{x}\|_{r}: \mathbf{x}\in {\mathbb{H}}^{n\times 1},\,\|\mathbf{x}\|_{r}=1\},\\
\|{\bf A}\|_{l}=\sup\{\|\mathbf{x}{\bf A}\|_{l}: \mathbf{x}\in  {\mathbb{H}}^{1\times n},\,\|\mathbf{x}\|_{l}=1\}.
\end{gather*}
Since  $\|\mathbf{x}\|_{r}= \|\mathbf{x}^{T}\|_{l}$, then $\|{\bf A}\|_{r}=\|{\bf A}\|_{l}$ for any ${\bf A}\in {\mathbb{H}}^{n\times n}$. Moreover, we define the norm for any ${\bf A}\in {\mathbb{H}}^{n\times n}$ by putting, $\|{\bf A}\|:=\|{\bf A}\|_{r}=\|{\bf A}\|_{l}$.
\begin{remark}In fact $\mathcal{H}_{r}$ and $\mathcal{H}_{l}$ are the right and left quaternion pre-Hilbert spaces, respectively. By introducing their completeness along the every fundamental quaternion unit we can define the right and left quaternion Banach spaces.
\end{remark}
We define the exponential of a quaternionic square matrix as well.
\begin{definition}\label{def:exp_matr} The exponential of a square matrix $ {\bf A} \in{\rm {\mathbb{H}}}^{n \times n}$ is the infinite sum
\begin{equation}\label{eq:exp_matr}e^{\bf A} =\sum_{n=0}^{\propto} \frac{{\bf A} ^n}{n!}.\end{equation}
\end{definition}
Since for any $k \in {\mathbb{N}}$,
$$
\| {\bf I}+{\bf A}+\frac{1}{2!}{\bf A}^{2}+\cdots+\frac{1}{k!}{\bf A}^{k}\|\leq \| {\bf I}\|+\|{\bf A}\|+\frac{1}{2!}\|{\bf A}^{2}\|+\cdots+\frac{1}{k!}\|{\bf A}^{k}\|\leq e^{\|{\bf A}\|},
$$
then the series (\ref{eq:exp_matr}) is converge absolutely.

From Definition \ref{def:exp_matr}, the following properties of the quaternionic matrix exponential evidently follow.
\begin{theorem}Let $ {\bf A} \in{\rm {\mathbb{H}}}^{n \times n}$.\begin{enumerate}
 \item[(i)] If ${\bf 0}$ is the $n \times n$ zero matrix and ${\bf I}$ is the $n \times n$ unit matrix, then $e^{\bf 0} ={\bf I}$ and $e^{\bf I} =e{\bf I}$.
      \item[(ii)] $e^{\bf A}= \lim_{k\mapsto \propto} \left({\bf I}+\frac{1}{k}{\bf A}\right)^{k}$.
\item[(iii)] $\left(e^{{\bf A}}\right)^{T} =e^{\left({\bf A}^{T}\right)}$, $\left(e^{{\bf A}}\right)^{*} =e^{\left({\bf A}^{*}\right)}$.
\item[(iv)] For all nonnegative integers $k$ holds ${\bf A}^{k}e^{\bf A}=e^{\bf A}{\bf A}^{k}$.
\item[(v)] If ${\bf A}{\bf B} ={\bf B}{\bf A}$, then  ${\bf A}e^{\bf B}=e^{\bf B}{\bf A}$ and $e^{\bf A}e^{\bf B}=e^{\bf B}e^{\bf A}=e^{{\bf A}+{\bf B}}$.
\item[(vi)]  If  $s, t\in{\mathbb{H}}$, then $e^{{\bf A}s}e^{{\bf A}t}=e^{{\bf A}(s+t)}$.
\item[(vii)] If $ {\bf A}$ is invertible, then $\left(e^{{\bf A}}\right)^{-1} =e^{\left({\bf A}^{-1}\right)}$.
\item[(viii)] If ${\bf D} = {\rm diag} [d_{1},\ldots, d_{n}]$, where $d_{i}\in {\mathbb{H}}$ for all $i=1,\ldots,n$, then $e^{\bf D} = {\rm diag} [e^{d_{1}},\ldots, e^{d_{n}}]$.
 \item[(ix)] If $ {\bf A}$ is diagonalizable with an invertible matrix $ {\bf P}$ and a diagonal matrix $ {\bf D}$ satisfying $ {\bf A} = {\bf P}{\bf D}{\bf P}^{-1}$, then $ e^{\bf A} = {\bf P}e^{\bf D}{\bf P}^{-1}$.
 \item[(x)]  If  ${\bf A}$ is Hermitian, then $\det e^{\bf A}={\rm tr}{\bf A}$.
  \end{enumerate}
\end{theorem}
\begin{proof} The proofs of the statements (i)-(ix) can be easy expanded from real matrix to quaternion matrices.

We only prove the point (x).
Since  ${\bf A}$ is Hermitian, then $e^{\bf A}$ is Hermitian as well, and we can define $\det e^{\bf A}$. By statement (ii) of this theorem,
$$\det e^{\bf A}= \det \left(\lim_{k\mapsto \propto} \left({\bf I}+\frac{1}{k}{\bf A}\right)^{k}\right)=
\lim_{k\mapsto \propto} \det\left({\bf I}+\frac{1}{k}{\bf A}\right)^{k}. $$
Using Definition \ref{def:ch_pol}, Theorem \ref{theor:char_polin} and  properties of limits, we obtain
$$\det e^{\bf A}=\lim_{k\mapsto \propto}\left(1+ \frac{{\rm tr}{\bf A}}{k}+O(k^{-2})\right)^{k}=\lim_{k\mapsto \propto}\left(1+ \frac{{\rm tr}{\bf A}}{k}\right)^{k}=e^{{\rm tr}{\bf A}}.
$$
Here  the  big $O$  notation is used.
\end{proof}

\subsection{Eigenvalues of quaternion matrices}
Due to the noncommutativity of quaternions, there are two
types of eigenvalues.
A quaternion $\lambda$ is said to be a left eigenvalue
of ${\rm {\bf A}} \in {\rm M}\left( {n,{\rm {\mathbb{H}}}}
\right)$ if
 \begin{equation}\label{eq:left1_eigen}
{\rm
{\bf A}} \cdot {\rm {\bf x}} = \lambda \cdot {\rm {\bf x}} \end{equation}
 for some nonzero  column-vector
${\rm {\bf x}}$ with quaternion components. Similarly, $\lambda$ is a right eigenvalue if
 \begin{equation}\label{eq:right_eigen}
{\rm {\bf A}} \cdot {\rm {\bf x}} = {\rm
{\bf x}} \cdot \lambda  \end{equation}
for some nonzero quaternionic column-vector
${\rm {\bf x}}$. Then, the  set  $\{\lambda\in {\mathbb{H}}|  {\bf A}  {\rm {\bf x}} = \lambda  {\rm {\bf x}},\,  {\bf x}\neq {\bf 0}  \in {\mathbb{H}}^{n}\} $ is  called  the
left  spectrum  of  ${\bf A}$,  denoted  by  $\sigma_{l}({\bf A})$.  The  right  spectrum  is  similarly  defined
  by putting, $\sigma_{r}({\bf A}):=\{\lambda\in {\mathbb{H}}|  {\bf A} {\rm {\bf x}} ={\rm {\bf x}}\lambda,\,  {\bf x}\neq {\bf 0}  \in {\mathbb{H}}^{n}\} $.

The theory on the left eigenvalues of quaternion matrices has been
investigated, in particular, in \cite{hu, so, wo}. The theory on the
right eigenvalues of quaternion matrices is more developed  (see, e.g. \cite{br,ma,ba,dra,zh,far}).
 We consider this is a natural consequence of the fact that quaternionic vector-columns form a right  vector space for which left eigenvalues from (\ref{eq:left1_eigen}) seem to be "exotic". Left eigenvalues   may appear natural in the equation
 \begin{equation}\label{eq:left_eigen}
 {\bf x}{\bf A}  = \lambda  {\rm {\bf x}}.
 \end{equation}
 Since  ${\bf x}{\bf A}  = \lambda  {\rm {\bf x}}$  if
and  only  if  ${\bf A}^{*} {\bf x}^{*}=  {\rm {\bf x}}^{*}\overline{\lambda} $, then the theory of such "natural" left eigenvalues from (\ref{eq:left_eigen}) be identical to the theory of right eigenvalues from (\ref{eq:right_eigen}). Similarly, the theory of right eigenvalues from ${\bf x}{\bf A}  = {\rm {\bf x}}\lambda  $ is identical to the theory of left eigenvalues from (\ref{eq:left1_eigen}).

Now, we present the some known results from the theory of right eigenvalues  that will be applied hereinafter. Due to the above, henceforth, we will avoid the "right" specification.

  In particular, it's well known that if   $\lambda$ is  a  nonreal  eigenvalue  of   ${\bf A}$,  so  is  any  element  in the
equivalence  class  containing   $[\lambda]$, i.e. $[\lambda]=\{x|x= u^{-1}\lambda u,\, u\in{\mathbb{H}},\, \|u\|=1\}$.

\begin{theorem}\cite{br} Any   quaternion  matrix  ${\rm {\bf A}} \in {\rm M}\left( {n,{\rm {\mathbb{H}}}}
\right)$  has  exactly  $n$    eigenvalues  which  are  complex  numbers  with
nonnegative  imaginary    parts.
\end{theorem}
\begin{proposition}\cite{ba}  Suppose that $\lambda_{1},\ldots, \lambda_{n}$ are distinct eigenvalues for ${\rm {\bf A}} \in {\rm M}\left( {n,{\rm {\mathbb{H}}}}
\right)$, no two of which are
conjugate, and let  ${\bf v}_{1},\ldots,{\bf v}_{n}$ be corresponding eigenvectors. Then  ${\bf v}_{1},\ldots,{\bf v}_{n}$ are (right) linearly independent.
\end{proposition}
Moreover, similarly to  the complex case, the following theorem can be proved.
\begin{theorem}\label{th:diag_gen}A matrix $ {\bf A} \in {\rm M}( n, {\mathbb{H}}
)$  is diagonalizable if and only if $ {\bf A}$ has
a  set of $n$ right-linearly independent eigenvectors. Furthermore, if $\lambda_{i}$,  ${\bf v}_{i}$, for $i = 1,\ldots, n$, are
eigenpairs of $ {\bf A}$ , then
   \begin{equation}\label{eq:diag_gen}
 {\bf A}={\bf P}{\bf D}{\bf P}^{-1},
 \end{equation}
where  ${\bf P}=\left[{\bf v}_{1},\ldots,{\bf v}_{n}\right]$,  ${\bf D}=diag \left[\lambda_{1},\ldots,\lambda_{n}\right]$.
\end{theorem}
 \begin{proof}
($\Rightarrow$)
Since  $ {\bf A}$ is diagonalizable, there exist an invertible matrix $ {\bf P}$ and a diagonal
matrix $ {\bf D}$ such that ${\bf A}={\bf P}{\bf D}{\bf P}^{-1}$. Then,
   \begin{equation}\label{eq:D}
 {\bf D}={\bf P}^{-1}{\bf A}{\bf P}.
 \end{equation}
Since $ {\bf D} \in {\rm M}( n, {\mathbb{H}}
)$  is diagonal, it has a right-linearly independent set of $n$ right eigenvectors given
by the column vectors of the identity matrix, i.e.,
  \begin{equation*}
 {\bf D}{\bf e}_{i}={\bf e}_{i}d_{ii},\,\,\,{\bf D}=diag\left[d_{11},\ldots,d_{nn}\right],\,\,\,{\bf I}=\left[{\bf e}_{1},\ldots,{\bf e}_{n}\right].
 \end{equation*}
So, the pair $d_{ii}$, ${\bf e}_{i}$ is an eigenvalue-eigenvector pair of  ${\bf D}$ for all $i = 1, \ldots, n$. Due to ${\bf e}_{i}d_{ii}=d_{ii}{\bf e}_{i}$ and using (\ref{eq:D}), we have
\begin{equation}\label{eq1:D}
{\bf e}_{i}d_{ii}=d_{ii}{\bf e}_{i}={\bf D}{\bf e}_{i}={\bf P}^{-1}{\bf A}{\bf P}{\bf e}_{i}.
\end{equation}
By multiplying  the extreme members of (\ref{eq1:D}) by ${\bf P}$ on the left, we obtain
\begin{equation*}
{\bf A}({\bf P}{\bf e}_{i})=({\bf P}{\bf e}_{i})d_{ii}.
\end{equation*}
It means that the vectors ${\bf v}_{i} = {\bf P}{\bf e}_{i}$ are right eigenvectors of  ${\bf A}$ with eigenvalue $d_{ii}$ for all $i = 1, \ldots, n$.
Since the matrix ${\bf P}=\left[{\bf v}_{1},\ldots,{\bf v}_{n}\right]$ is invertible, then, by Theorem \ref{theorem:inver_equiv}, the eigenvectors ${\bf v}_{1},\ldots,{\bf v}_{n}$ is right-linearly independent.
($\Leftarrow$) Let  $\lambda_{i}$,  ${\bf v}_{i}$,  be eigenvalue-eigenvector pairs of  ${\bf  A}$ for $i = 1,\ldots, n$. Consider the matrix ${\bf P}=\left[{\bf v}_{1},\ldots,{\bf v}_{n}\right]$.
Computing the product, we obtain
\[
{\bf  A}{\bf P}={\bf  A}\left[{\bf v}_{1},\ldots,{\bf v}_{n}\right]=\left[{\bf  A}{\bf v}_{1},\ldots,{\bf  A}{\bf v}_{n}\right]=\left[{\bf v}_{1}\lambda_{1},\ldots,{\bf v}_{n}\lambda_{n}\right].
\]
 Since the eigenvector set  $\{ {\bf v}_{1},\ldots,{\bf v}_{n} \}$ is right-linearly independent, then, by Theorem \ref{theorem:inver_equiv},  ${\bf P}$ is invertible. There exists ${\bf P}^{-1}$, and
\[
{\bf P}^{-1}{\bf  A}{\bf P}={\bf P}^{-1}\left[{\bf v}_{1}\lambda_{1},\ldots,{\bf v}_{n}\lambda_{n}\right]=\left[{\bf P}^{-1}{\bf v}_{1}\lambda_{1},\ldots,{\bf P}^{-1}{\bf v}_{n}\lambda_{n}\right].
\]
Since ${\bf P}^{-1}{\bf P}={\bf I}$, then ${\bf P}^{-1}{\bf v}_{i}={\bf e}_{i}$ for all $i = 1, \ldots, n$. So,
\[
{\bf P}^{-1}{\bf  A}{\bf P}=\left[{\bf e}_{1}\lambda_{1},\ldots,{\bf e}_{n}\lambda_{n}\right]=diag\left[\lambda_{1},\ldots,\lambda_{n}\right]
\]
Denoting ${\bf  D} = diag\left[\lambda_{1},\ldots,\lambda_{n}\right]$,
we conclude that  ${\bf P}^{-1}{\bf A}{\bf P}={\bf D}$, or equivalently, ${\bf A}={\bf P}{\bf D}{\bf P}^{-1}$.
\end{proof}
\begin{Corollary}\cite{ba}If ${\rm {\bf A}} \in {\rm M}\left( {n,{\rm {\mathbb{H}}}}
\right)$ has $n$ non-conjugate eigenvalues, then it can be diagonalized in the sense
that there is a ${\bf P} \in GL_{n}({\mathbb H})$ for which ${\bf P}{\bf A}{\bf P}^{-1}$ is diagonal.
\end{Corollary}
Those  eigenvalues $h_{1} + k_{1}{\bf  i},\ldots, h_{n} + k_{n}{\bf  i}$, where $k_{t}\geq0$ and
$h_{t}, k_{t}\in {\mathbb{R}}$  for all $t = 1,\ldots, n$, are  said  to  be  the  standard  eigenvalues  of  ${\bf A}$.
\begin{theorem}\cite{br} Let $ {\bf A} \in {\rm M}\left( n, {\mathbb{H}}
\right)$. Then there exists a unitary matrix ${\bf U}$ such that ${\bf U}^{*}{\bf AU}$ is an upper
triangular matrix with diagonal entries $h_{1} + k_{1}{\bf  i},\ldots, h_{n} + k_{n}{\bf  i}$ which are the standard eigenvalues of  ${\bf A}$.
\end{theorem}
\begin{Corollary}\cite{zh}\label{th:diag_spec}
Let ${\rm {\bf A}} \in {\rm M}\left( {n,{\rm {\mathbb{H}}}}
\right)$ with the standard eigenvalues $h_{1} + k_{1}{\bf  i},\ldots, h_{n} + k_{n}{\bf i}$.
Then
$
\sigma_{r}=[h_{1} + k_{1}{\bf  i}]\cup\cdots\cup[h_{n} + k_{n}{\bf  i}].
$
\end{Corollary}
\begin{Corollary}\cite{zh}\label{th:diag_norm} ${\rm {\bf A}} \in {\rm M}\left( {n,{\rm {\mathbb{H}}}}
\right)$ is  normal  if  and  only  if  there  exists  an  unitary  matrix
${\bf U}\in {\rm M}\left( {n,{\rm {\mathbb{H}}}}
\right)$  such  that
   \begin{equation}\label{eq:diag_norm}
   {\bf U}^{*}{\bf A}{\bf U}={\rm diag}\{\lambda_{1},\dots,\lambda_{n}\},
   \end{equation}
where $\lambda_{i}=h_{i}+k_{i}{\bf i}\in {\mathbb C}$ is standard eigenvalues for all $i=1,\ldots,n$. ${\rm {\bf A}}$ is  Hermitian if  and  only  if  $k_{i}=0$ and $\lambda_{i}=h_{i}\in {\mathbb R}$.
\end{Corollary}
From Theorem \ref{th:diag_gen} and Corollaries \ref{th:diag_spec} and \ref{th:diag_norm} the following proposition evidently follows.
\begin{theorem}\label{th:sim_norm} An arbitrary matrix ${\rm {\bf A}} \in {\rm M}\left( {n,{\rm {\mathbb{H}}}}
\right)$ is diagonalizable if and only if it is similar to some normal matrix.
\end{theorem}
\begin{proof}($\Rightarrow$)
Let $ {\bf A} \in {\rm M}( n, {\mathbb{H}}
)$  is diagonalizable. Then by (\ref{eq:diag_gen}), there exists an invertible matrix ${\bf P}$ such that
 ${\bf A}={\bf P}{\bf D}{\bf P}^{-1}$, where ${\bf D}={\rm diag}\{\lambda_{1},\dots,\lambda_{n}\}$ and $\lambda_{i}=h_{i}+k_{i}{\bf i}\in {\mathbb C}$ is standard eigenvalues for all $i=1,\ldots,n$.
Let $ {\bf N} \in {\rm M}( n, {\mathbb{H}})$ be a normal matrix such that by (\ref{eq:diag_norm})
   ${\bf U}{\bf N}{\bf U}^{*}={\rm diag}\{\lambda_{1},\dots,\lambda_{n}\}$, where ${\bf U}$ is some unitary matrix. Then, we have
\[  {\bf A}={\bf P}{\bf U}{\bf N}{\bf U}^{*}{\bf P}^{-1} =\left({\bf P}{\bf U}\right){\bf N}\left({\bf P}{\bf U}\right)^{-1}.
  \]
 ($\Leftarrow$) Let ${\rm {\bf A}} $ is similar to some normal matrix ${\bf N} \in {\rm M}( n, {\mathbb{H}})$. It means  there exists an invertible matrix ${\bf T}$ such that
 ${\bf A}={\bf T}^{-1}{\bf N}{\bf T}$. Since ${\bf N}={\bf U}^{*}{\rm diag}\{\lambda_{1},\dots,\lambda_{n}\}{\bf U}$, then
 \[{\bf A}={\bf T}^{-1}{\bf U}^{*}{\rm diag}\{\lambda_{1},\dots,\lambda_{n}\}{\bf U}{\bf T}=({\bf U}{\bf T})^{-1}{\rm diag}\{\lambda_{1},\dots,\lambda_{n}\}({\bf U}{\bf T}). \]
\end{proof}
Right (\ref{eq:right_eigen}) and left (\ref{eq:left1_eigen}) eigenvalues are in general unrelated \cite{fa}, but it is not  for Hermitian matrices.
Suppose ${\rm {\bf A}} \in {\rm M}\left( {n, {\mathbb{H}}}\right)$
is Hermitian and $\lambda \in {\rm {\mathbb {R}}}$ is its right
eigenvalue, then ${\rm {\bf A}} \cdot {\rm {\bf x}} = {\rm {\bf
x}} \cdot \lambda = \lambda \cdot {\rm {\bf x}}$. This means that
all right eigenvalues of a Hermitian matrix are its left
eigenvalues as well. For real left eigenvalues, $\lambda \in {\rm
{\mathbb {R}}}$, the matrix $\lambda {\rm {\bf I}} - {\rm {\bf
A}}$ is Hermitian.
\begin{definition}\label{def:ch_pol}
If $t \in {\rm {\mathbb {R}}}$, then for a Hermitian matrix ${\rm
{\bf A}}$ the polynomial $p_{{\rm {\bf A}}}\left( {t} \right) =
\det \left( {t{\rm {\bf I}} - {\rm {\bf A}}} \right)$ is said to
be the characteristic polynomial of ${\rm {\bf A}}$.
\end{definition}
The roots of the characteristic polynomial of a Hermitian matrix
are its real left eigenvalues, which are its right eigenvalues as
well. We can prove the following theorem by analogy to the
commutative case (see, e.g. \cite{la}).

\begin{theorem}\label{theor:char_polin}
If ${\rm {\bf A}} \in {\rm M}\left( {n,{\rm {\mathbb{H}}}}
\right)$ is Hermitian, then $p_{{\rm {\bf A}}}\left( {t} \right) =
t^{n} - d_{1} t^{n - 1} + d_{2} t^{n - 2} - \ldots + \left( { - 1}
\right)^{n}d_{n}$, where $d_{1}={\rm tr}{\bf A} $, $d_{k} $ is the sum of principle minors
of ${\rm {\bf A}}$ of order $k$ for all $1 < k < n$, and $d_{n}=\det
{\rm {\bf A}}$.
\end{theorem}

\subsection{An algorithm for obtaining eigenvalues of normal quaternion matrices}
Using Theorem \ref{theor:char_polin} we can obtain   eigenvalues not only  of   Hermitian matrices but  of normal matrices as well. Moreover, if it is known a normal matrix which is similar to some quaternionic matrix  ${\bf A}$, then we also can find its   eigenvalues.
Let us derive eigenvalues and eigenvectors of a quaternion normal matrix ${\rm {\bf N}} \in {\rm M}\left( {n, {\mathbb{H}}}\right)$. Then the following algorithm can be considered.

Step 1: Find eigenvalues of the corresponding Hermitian matrix ${\bf N}^{*}{\bf N}$ by solving  roots of its characteristic polynomial,
$\det \left( {{\bf N}^{*}{\bf N}} - \lambda{\rm {\bf I}} \right)=0$.  Let the eigenvalues of ${\bf N}^{*}{\bf N}$ be $\lambda_{1}, \ldots, \lambda_{n}$, where $\lambda_{i}\in {\mathbb{R}}$ for all $i=1,\ldots,n$. The associated
eigenvectors are computed as the solutions to the equation $\left( {{\bf N}^{*}{\bf N}} - \lambda{\rm {\bf I}} \right){\bf v} = 0$.

Step 2: Conduct the unitary diagonalization of  ${\bf N}^{*}{\bf N}$. We determine all its own subspaces and choose  orthonormal basis in each of them for example by the Gram-Schmidt orthogonalization process. We  obtain orthonormal system of vectors ${\bf u}_{1}, \ldots,{\bf u}_{n}$. Since the matrix is diagonalizable, the union of all these bases is a basis of the whole space.
Construct the unitary matrix ${\bf U}$ which columns are ${\bf u}_{1}, \ldots,{\bf u}_{n}$.

Step 4: Since the unitary matrix ${\bf U}$ is applicable for diagonalization of ${\bf N}$ as well, we can find a diagonal matrix ${\bf D}={\bf U}^{*} {\bf N}{\bf U}$, where ${\bf D}={\rm diag}\{\mu_{1},\dots,\mu_{n}\}$, $\mu_{i}=h_{i}+k_{i}{\bf i}\in {\mathbb C}$ is standard eigenvalues for all $i=1,\ldots,n$ such that $\lambda_{i}=\overline{\mu}_{i}\mu_{i}$.

 Step 5: Let for  some quaternionic matrix  ${\bf A}$    it  be given an invertible matrix  ${\bf T}$ such that ${\bf A}={\bf T}{\bf N}{\bf T}^{-1}$. Then by Theorem \ref{th:sim_norm} $\mu_{i}$ for all $i=1,\ldots,n$ are also eigenvalues of  ${\bf A}$, and by Theorem \ref{th:diag_gen}, columns of  ${\bf T}{\bf U}$ are its corresponding eigenvectors.

We illustrate this algorithm  by the following example.

\begin{ex}\label{ex1}
Consider the normal matrix ${\bf N}=\begin{bmatrix}2 & 0 & i+j\\
 0 & i & 0\\
                                                     i -j & 0&  2
\end{bmatrix}.$ Its corresponding Hermitian matrix is
\[{\bf M}={\bf N}^{*}{\bf N}=\begin{bmatrix}6 & 0 & 4j\\
 0 & 1 & 0\\
                                                      -4j & 0&  6
\end{bmatrix}.\]
Find the eigenvalues of ${\bf M}$ which are the roots of the characteristic
polynomial
\[
p(\lambda)=\det\begin{bmatrix}\lambda-6 & 0 & 4j\\
 0 & \lambda-1 & 0\\
                                                      -4j & 0&  \lambda-6
\end{bmatrix}=\lambda^{3}-13\lambda^{2}+32\lambda-20\,\,\Rightarrow\,\,\begin{cases}
                                                           \lambda_{1}=10, \\
                                                           \lambda_{2}=1,\\
                                                           \lambda_{3}=2.
                                                         \end{cases}
\]
By computing the associated eigenvectors and after their orthonormalization, we obtain the unitary matrix ${\bf U}$ whose columns are this eigenvectors.
\[{\bf U}=\begin{bmatrix}0.5 -0.5j & 0 & 0.5 +0.5j\\
 0 & 1 & 0\\
                                                      0.5 +0.5j & 0&  0.5 -0.5j
\end{bmatrix}.\] Finally, we have
\begin{equation}\label{eq:eigenval}{\bf D}={\bf U}^{*}{\bf N}{\bf U}=\begin{bmatrix}1+i & 0 & 0\\
 0 & i & 0\\
                                                      0 & 0&  3+i
\end{bmatrix}\,\,\Rightarrow\,\,\begin{cases}
                                                           \mu_{1}=1+i, \\
                                                           \mu_{2}=i,\\
                                                           \mu_{3}=3+i.
                                                         \end{cases}\end{equation}
Moreover, consider the matrix
\[{\bf A}=\begin{bmatrix}1- 2.5i -0.5j+k & 4+3j+2.5k & 2- 2i -j-2.5k\\
 1.5-i -j-0.5k & 2+1.5i -3j+3k & 2+2.5i+j-k\\
                                                      0.5-i +j-0.5k & 3-i -0.5j&  1+i -1.5j-2k
\end{bmatrix}.\]
There exists the matrix
\[{\bf T}=\begin{bmatrix}-k & j & 2\\
 i & k & i\\
                                                      -j & 1&  i
\end{bmatrix}\] such that its the inverse is
\[{\bf T}^{-1}=\begin{bmatrix}-0.5+0.5k & -0.5i+j &-0.5i\\
 0.5i-0.5j & -1.5 & 0.5+k\\
                                                      0 & -0.5i+0.5j&  -0.5i-0.5j
\end{bmatrix},\]
and ${\bf A}={\bf T}{\bf N}{\bf T}^{-1}$. Then by Theorem \ref{th:sim_norm} the eigenvalues of  ${\bf A}$ are (\ref{eq:eigenval}), and its corresponding eigenvectors are the columns of the matrix
\[{\bf T}{\bf U}=\begin{bmatrix}1-0.5i+j-0.5k & j & 1+0.5i-j-0.5k\\
 i & k & i\\
                                                     -0.5+0.5i-0.5j+0.5k & 1&  -0.5+0.5i-0.5j-0.5k
\end{bmatrix}.\]
\end{ex}

\section{Systems of quaternion linear differential  equations}
\subsection{Definitions}
Consider a matrix valued function ${\bf
A}(t) = \left({ a}_{ij}(t)\right)\in{\rm {\mathbb{H}}}^{n \times n}\otimes {\mathbb{R}}$, where ${ a}_{ij}(t)$ are quaternion-valued functions with the real variable $t$ for all $i,j={1,\ldots,n}$.
 Then
 \begin{equation*}
\frac{{\rm d}{\bf A}(t)}{{\rm d}t}=\left(\frac{{\rm d}{ a}_{ij}(t)}{{\rm d}t}\right)_{n\times n},\;\;\;\;
\int_a^b {\bf
A}(t){{\rm d}t}= \left(\int_a^b { a}_{ij}(t){{\rm d}t}\right)_{n\times n}.
\end{equation*}
Over the quaternion skew field, we can consider the following systems of  linear differential  equations.
 \begin{definition}An $n \times n$ first order right  linear quaternion differential system is the equation
 \begin{equation}\label{eq:right_sys_dif}
 {\bf x}'= {\bf A}(t){\bf x}+{\bf b}(t),
\end{equation}
where ${\bf
A}(t) \in{\rm {\mathbb{H}}}^{n \times n}\otimes {\mathbb{R}}$ is the  coefficient matrix, ${\bf b}(t)=\begin{bmatrix}
  b_{1}(t) \\
   \vdots\\
  b_{n}(t)
\end{bmatrix}\in{\mathbb{H}}^{n \times 1}\otimes {\mathbb{R}}$ is the given column-vector,  ${\bf x}(t)=\begin{bmatrix}
  x_{1}(t) \\
   \vdots\\
  x_{n}(t)
\end{bmatrix}$ is the unknown column-vector.
An $n \times n$ first order left linear quaternion differential system is the equation
 \begin{equation}\label{eq:left_sys_dif}
 {\bf x}'= {\bf x}{\bf A}(t)+{\bf b}(t),
\end{equation}
where ${\bf b}(t)=\left(
  b_{1}(t)
   \cdots
  b_{n}(t)
\right)\in{\mathbb{H}}^{1 \times n}\otimes {\mathbb{R}}$ is the given row-vector,  ${\bf x}(t)=\left(
  x_{1}(t)
   \cdots
  x_{n}(t)
\right)$ is the unknown row-vector.

The systems (\ref{eq:right_sys_dif}) and (\ref{eq:left_sys_dif}) are called nonhomogeneous when there exists $t \in {\mathbb{R}}$ such that ${\bf b}(t) \neq 0$, and homogeneous  when the source vector ${\bf b} \equiv 0$, i.e., respectively,
 \begin{gather}\label{eq:right_sys_dif_hom}{\bf x}'= {\bf A}(t){\bf x},\\\label{eq:left_sys_dif_hom}
{\bf x}'= {\bf x}{\bf A}(t).
\end{gather}
\end{definition}
 \begin{remark}
By the definition of the matrix-vector product, Eq. (\ref{eq:right_sys_dif}) can be written as
\[\left\{\begin{aligned}x^{'}_1 & =a_{11}(t)x_{1}&+&\cdots &+& a_{1n}(t)x_{n}&+& b_{1}(t), \\
\vdots\\
x^{'}_n & =a_{n1}(t)x_{1}&+&\cdots &+& a_{nn}(t)x_{n}&+& b_{n}(t), \end{aligned}\right.
\]
and Eq. (\ref{eq:left_sys_dif}) can be written as
\[\left\{\begin{aligned}x^{'}_1 & =&x_{1}a_{11}(t)&+&\cdots &+& x_{n}a_{n1}(t)&+& b_{1}(t), \\
\vdots\\
x^{'}_n & =&x_{1}a_{1n}(t)&+&\cdots &+& x_{n}a_{nn}(t) &+& b_{n}(t). \end{aligned}\right.
\]
\end{remark}
\begin{definition}
Solutions of the linear differential systems (\ref{eq:right_sys_dif}) and (\ref{eq:left_sys_dif}) are, respectively, column-vector   and row-vector valued functions  {\bf x}(t) that satisfy every differential equation in the systems.
\end{definition}
\begin{definition}  \textbf{Initial Value Problems} for right and left quaternion linear differential systems are, respectively,
the following: Given an  matrix valued function ${\bf
A}(t) = \left({\bf a}_{ij}(t)\right)\in{\rm {\mathbb{H}}}^{n \times n}\otimes {\mathbb{R}}$, and a quaternion vector  valued function ${\bf b}(t)$,
a real constant $t_0$, and a vector ${\bf x_{0}}$, find  a quaternion  vector valued function ${\bf x}(t)$ that is a solution of
\begin{equation}\label{eq:right_sys_dif_cosh}
{\bf x}'= {\bf A}(t){\bf x}+{\bf b}(t),\,\,\,{\bf x}(t_{0})={\bf x}_{0},
\end{equation}
or
\begin{equation}\label{eq:left_sys_dif_cosh}
{\bf x}'= {\bf x}{\bf A}(t)+{\bf b}(t),\,\,\,{\bf x}(t_{0})={\bf x}_{0}.
\end{equation}
\end{definition}Similarly to real linear
differential equations, we can proved the following theorem  about existence and uniqueness of solutions to the initial value problems.
 \begin{theorem}\label{th:uniq_init_sol}
 If the functions ${\bf
A}(t)$ and ${\bf
b}(t)$ are continuous on
an open interval $I \in{\mathbb{H}}$, and if ${\bf x}_{0}$ is any constant vector (column or row, respectively) and $t_{0}$ is any constant in $I$, then
there exist only one function ${\bf x}(t)$, defined an interval $\tilde{I}\in I$ with $t_{0} \in \tilde{I}$, that is a solution of the initial
value problems (\ref{eq:right_sys_dif_cosh}) or (\ref{eq:left_sys_dif_cosh}), respectively.
\end{theorem}

 \subsection{General solutions of homogenous systems.}

\begin{definition}
A set of quaternion column-vector functions $\{{\bf x}_{1}(t),\ldots, {\bf x}_{n}(t)\}$ is called right linearly
dependent on an interval $I \in {\mathbb R}$ if for all $t \in I$ there exist constant quaternions $q_{1},\ldots, q_{n}$, ($q_{i} \in {\mathbb H}$ for all $i=1,\ldots,n$), not all of
them zero, such that it holds,
\begin{equation*}
{\bf x}_{1}(t)q_{1}+\cdots+{\bf x}_{n}(t)q_{n}={\bf 0}.
\end{equation*}
Similarly, a set of quaternion row-vector  functions $\{{\bf x}_{1}(t),\ldots, {\bf x}_{n}(t)\}$ is called left linearly
dependent on an interval $I \in {\mathbb R}$ if under the same conditions,
\begin{equation*}
q_{1}{\bf x}_{1}(t)+\cdots+q_{n}{\bf x}_{n}(t)={\bf 0}.
\end{equation*}
These sets are called right (left) linearly independent on $I$ if they are not right (left)
linearly dependent.
\end{definition}
 \begin{theorem}\label{th:riht_lin_comb_sol} If the column-vector valued functions ${\bf x}_{1}$, ${\bf x}_{2}$ are solutions of the homogenies system  (\ref{eq:right_sys_dif_hom}), i.e.
 ${\bf x}^{'}_{1} = {\bf A}(t) {\bf x}_{1}$ and ${\bf x}'_{2} = {\bf A}(t) {\bf x}_{2}$, then any  right linear combination ${\bf x} = {\bf x}_{1}a + {\bf x}_{2}b$,
for all $a, b \in {\mathbb H}$ is also a solution of (\ref{eq:right_sys_dif_hom}).
 \end{theorem}
 \begin{proof}Indeed, since the derivative of a vector valued
function is a linear operation, we get
\begin{equation*}
{\bf x}^{'} = \left({\bf x}_{1}a + {\bf x}_{2}b\right)^{'}={\bf x}^{'}_{1}a + {\bf x}^{'}_{2}b.
\end{equation*}
Replacing the differential equation on the right-hand side above,
\begin{equation*}
{\bf x}^{'} ={\bf A}(t) {\bf x}_{1}a + {\bf A}(t) {\bf x}_{2}b.
\end{equation*}
Since the matrix-vector product is a linear operation, then ${\bf A}{\bf x}_{1}a + {\bf A}{\bf x}_{2}b={\bf A}({\bf x}_{1}a+ {\bf x}_{2}b)$. Hence,
\begin{equation*}
{\bf x}^{'} ={\bf A}({\bf x}_{1}a+ {\bf x}_{2}b)={\bf A}{\bf x}.
\end{equation*}
This establishes the theorem.
\end{proof}
The following theorem can be proved similarly.
 \begin{theorem}If the row-vector valued functions ${\bf x}_{1}$, ${\bf x}_{2}$ are solutions of the homogenies system  (\ref{eq:left_sys_dif_hom}), i.e.
 ${\bf x}^{'}_{1} =  {\bf x}_{1}{\bf A}(t)$ and ${\bf x}^{'}_{2} =  {\bf x}_{1}{\bf A}(t)$, then any left linear combination ${\bf x} = a{\bf x}_{1} + b{\bf x}_{2}$,
for all $a, b \in {\mathbb H}$ is also solution of (\ref{eq:left_sys_dif_hom}).
 \end{theorem}
 We have the following theorems about right an left homogenies systems.
  \begin{theorem}\label{th:right_sol}
 If $\{{\bf x}_{1}, \cdots, {\bf x}_{n}\}$ is a right linearly independent set of
solutions of (\ref{eq:right_sys_dif_hom}), where ${\bf A}$ is a continuous matrix valued function, then
there exist constant quaternions $q_{1},\cdots, q_{n}$, ($q_{i} \in {\mathbb H}$ for all $i=1,\ldots,n$) such that every solution ${\bf x}$ of (\ref{eq:right_sys_dif_hom}) can be written as the right linear combination
\begin{equation}
{\bf x}(t)={\bf x}_{1}(t)q_{1}+\cdots+{\bf x}_{n}(t)q_{n}.
\end{equation}
 \end{theorem}
\begin{proof} By Theorem \ref{th:riht_lin_comb_sol}, the right linear combination ${\bf x}(t)={\bf x}_{1}(t)q_{1}+\cdots+{\bf x}_{n}(t)q_{n}$  is  a solution of (\ref{eq:right_sys_dif_hom}) as well. We now must prove that, in the case
that $\{{\bf x}_{1}, \cdots, {\bf x}_{n}\}$ is right linearly independent, every solution of (\ref{eq:right_sys_dif_hom}) is
included in this linear combination.

Let ${\bf x}$ be any solution of the differential equation (\ref{eq:right_sys_dif_hom}). Due to uniqueness statement in
Theorem \ref{th:uniq_init_sol}, this is the unique solution that at $t_{0}$ takes the value ${\bf x}(t_{0})$. This
means that the initial data ${\bf x}(t_{0})$ parameterizes all solutions to the differential equation (\ref{eq:right_sys_dif_hom}). Then, we shall find the constants $q_{1},\cdots, q_{n}$ as solutions of the algebraic linear system,
\begin{equation*}
{\bf x}(t_{0})={\bf x}_{1}(t_{0})q_{1}+\cdots+{\bf x}_{n}(t_{0})q_{n}.
\end{equation*}
Introducing the notation
\begin{equation*}
{\bf X}(t)=\left[{\bf x}_{1}(t),\ldots,{\bf x}_{n}(t)\right],\,\,\,{\bf q}=\begin{bmatrix}
  q_{1} \\
   \vdots\\
  q_{n}
\end{bmatrix},
\end{equation*}
the algebraic linear system has the form
\begin{equation*}
{\bf x}(t_{0})={\bf X}(t_{0}){\bf q}.
\end{equation*}
This algebraic system has a unique solution ${\bf q}$ for every source ${\bf x}(t_{0})$ when the matrix ${\bf X}(t_{0})$
is invertible. By Theorem \ref{theorem:deter_inver}, the necessary and sufficient condition of invertibility
of  ${\bf X}(t_{0}) \in {\rm M}(n,{{\rm {\mathbb{H}}}})$ is
${\rm{ddet}} {\bf X}(t_{0}) \ne 0$. By Theorem \ref{theorem:crit_depen}, it is equivalent that $\{{\bf x}_{1}, \cdots, {\bf x}_{n}\}$ is right linearly independent.\end{proof}
\begin{theorem}
 If $\{{\bf x}_{1}, \cdots, {\bf x}_{n}\}$ is a left linearly independent set of
solutions of (\ref{eq:left_sys_dif_hom}), where ${\bf A}$ is a continuous matrix valued function, then
there exist constant quaternions $q_{1},\cdots, q_{n}$, ($q_{i} \in {\mathbb H}$ for all $i=1,\ldots,n$) such that every solution ${\bf x}$ of (\ref{eq:left_sys_dif_hom}) can be written as the left linear combination
\begin{equation}
{\bf x}(t)=q_{1}{\bf x}_{1}(t)+\cdots+q_{n}{\bf x}_{n}(t).
\end{equation}
 \end{theorem}
\begin{proof}
The proof is similar to the proof of Theorem \ref{th:right_sol}.
\end{proof}
So, we obtain the following definitions.
\begin{definition}Let ${\bf x}_{i}(t)\in {\mathbb H}^{n\times 1}\otimes {\mathbb{R}}$ be a quaternion column-vector valued function for all $i=1,\ldots,n.$ \begin{enumerate}
\item The set  $\{{\bf x}_{1}, \ldots, {\bf x}_{n}\}$ is a fundamental set of solutions of (\ref{eq:right_sys_dif_hom}) if it is a set of right-linearly independent column-vectors which are solutions of (\ref{eq:right_sys_dif_hom}).
\item The general solution of the homogeneous equation (\ref{eq:right_sys_dif_hom}) denotes any quaternion column-vector valued
function ${\bf x}_{gen}$ that can be written as a right linear combination
\begin{equation*}
{\bf x}_{gen}(t)={\bf x}_{1}(t)q_{1}+\cdots+{\bf x}_{n}(t)q_{n},
\end{equation*}
where  $\{{\bf x}_{1}, \ldots, {\bf x}_{n}\}$ is a fundamental set of solutions of (\ref{eq:right_sys_dif_hom}), while $q_{1}, \ldots, q_{n}$ are arbitrary quaternion constants.
\item  A solution matrix $
{\bf X}_{r}(t)=\left[{\bf x}_{1}(t),\ldots,{\bf x}_{n}(t)\right]$
 is called a fundamental matrix of (\ref{eq:right_sys_dif_hom}) if the set $\{{\bf x}_{1}, \ldots, {\bf x}_{n}\}$ is a fundamental set.
                              \end{enumerate}
\end{definition}
\begin{definition}Let ${\bf x}_{i}(t)\in {\mathbb H}^{1\times n}\otimes {\mathbb{R}}$ be a quaternion row-vector valued function for all $i=1,\ldots,n.$
\begin{enumerate}
\item The set  $\{{\bf x}_{1}, \ldots, {\bf x}_{n}\}$  is a fundamental set of solutions of (\ref{eq:left_sys_dif_hom}) if it is a set of left-linearly independent row-vectors which are solutions of (\ref{eq:left_sys_dif_hom}).
\item The general solution of  (\ref{eq:left_sys_dif_hom}) denotes any quaternion row-vector valued
function ${\bf x}_{gen}$ that can be written as a left linear combination
\begin{equation*}
{\bf x}_{gen}(t)=q_{1}{\bf x}_{1}(t)+\cdots+q_{n}{\bf x}_{n}(t),
\end{equation*}
where $\{{\bf x}_{1}, \ldots, {\bf x}_{n}\}$ is a fundamental set of solutions of (\ref{eq:left_sys_dif_hom}), while $q_{1}, \ldots, q_{n}$ are arbitrary quaternion constants.
\item A solution matrix ${\bf X}_{l}(t)=\begin{bmatrix}{\bf x}_{1}(t)\\\vdots\\{\bf x}_{n}(t)\end{bmatrix}$
 is called a fundamental matrix of (\ref{eq:left_sys_dif_hom}) if the set $\{{\bf x}_{1}, \ldots, {\bf x}_{n}\}$ is a fundamental set.
                              \end{enumerate}
\end{definition}

\begin{remark}From the above definitions, it follows that the general solutions of (\ref{eq:right_sys_dif_hom}) are (\ref{eq:left_sys_dif_hom}), respectively, can be represented as
\begin{gather}\label{eq:gen_right}{\bf x}_{gen}(t)={\bf X}_{r}(t){\bf q},\\
\label{eq:gen_left}{\bf x}_{gen}(t)={\bf q}\,{\bf X}_{l}(t).
\end{gather}
Moreover, for given Initial Value Problems, ${\bf x}(t_{0})={\bf x}^{0}$, we have, respectively,
\begin{gather}\label{eq:gen_right_init}{\bf x}_{gen}(t)={\bf X}_{r}(t){\bf X}_{r}^{-1}(t_{0}){\bf x}^{0},\\
\label{eq:gen_left_init}{\bf x}_{gen}(t)={\bf x}^{0}{\bf X}_{l}^{-1}(t_{0}){\bf X}_{l}(t).
\end{gather}
\end{remark}

 \subsection{General solutions of non-homogenous systems.}
Firstly, we note that by simple checking can be proved the following lemma.
\begin{lemma} If ${\bf u}^{NH}$ is a solution of the right nonhomogenous system (\ref{eq:right_sys_dif}), and  ${\bf v}^{H}$ is a solution of the right homogenous system (\ref{eq:right_sys_dif_hom}), then ${\bf u}^{NH}+{\bf v}^{H}$ is a solution of (\ref{eq:right_sys_dif}). Similarly, we have to left systems (\ref{eq:left_sys_dif})-(\ref{eq:left_sys_dif_hom}).
\end{lemma}
Using the representations (\ref{eq:gen_right}) and (\ref{eq:gen_left}), respectively, for  ${\bf v}^{H}$, we evidently obtain the following theorem.
\begin{theorem}For any solutions ${\bf x}^{NH}$ of nonhomogenous systems (\ref{eq:right_sys_dif}) and (\ref{eq:left_sys_dif}), they can be expressed, respectively, as follows
\begin{equation*}{\bf x}^{NH}(t)={\bf X}_{r}(t){\bf q}+{\bf u}^{NH},
\end{equation*}
where ${\bf u}^{NH}\in {\mathbb H}^{n \times 1}$ is  an arbitrary solution of (\ref{eq:right_sys_dif}),  and ${\bf q}\in {\mathbb H}^{n \times 1}$ is a constant quaternionic column-vector;
\begin{equation*}{\bf x}^{NH}(t)={\bf q}\,{\bf X}_{l}(t)+{\bf u}^{NH},
\end{equation*}
where ${\bf u}^{NH}\in {\mathbb H}^{ 1\times n}$ is  an arbitrary solution of (\ref{eq:left_sys_dif}),  and ${\bf q}\in {\mathbb H}^{1\times n }$ is a constant quaternionic row-vector.
\end{theorem}
The following theorem represent the general solution of the right nonhomogenous system (\ref{eq:right_sys_dif}).
\begin{theorem}\cite{kou3}The general solution of
(\ref{eq:right_sys_dif}) is given by
\begin{equation*}
{\bf x}^{NH}(t)={\bf X}_{r}(t){\bf q}+{\bf X}_{r}(t)\int_{t_{0}}^{t}{\bf X}^{-1}_{r}(s){\bf b}(s)dt,
\end{equation*}
where $t_{0}\in I\in {\mathbb R}$, ${\bf q}\in {\mathbb H}^{n \times 1}$ is a constant quaternionic column-vector.
\end{theorem}
This theorem is completely proved in \cite{kou3} using the  Chen determinant. We prove the following theorem about the general solution of the left nonhomogenous system (\ref{eq:left_sys_dif}) within the framework of the theory of column-row determinants.

\begin{theorem}The general solution of
(\ref{eq:left_sys_dif}) is given by
\begin{equation}
\label{eq:end_gen_left_non}
{\bf x}^{NH}_{gen}(t)={\bf q}{\bf X}_{l}(t)+\int_{t_{0}}^{t}{\bf b}(s){\bf X}^{-1}_{l}(s)dt\,{\bf X}_{l}(t),
\end{equation}
where $t_{0}\in I\subset {\mathbb R}$, ${\bf q}\in {\mathbb H}^{ 1\times n}$ is a constant quaternionic row-vector.
\end{theorem}
\begin{proof}By (\ref{eq:gen_left}) the general solution of (\ref{eq:left_sys_dif_hom}) is
$
 {\bf x}_{gen}(t)={\bf q}\,{\bf X}_{l}(t),
$
where ${\bf q}\in {\mathbb H}^{1 \times n}$ is a constant quaternionic row-vector. Let us find a solution of
 \begin{equation}\label{eq:left_sys_dif1}
 {\bf x}'= {\bf x}(t){\bf A}(t)+{\bf b}(t),
\end{equation}
 in the form,
\begin{equation}
\label{eq:left_sys_dif_nonhom1}
 {\bf x}(t)={\bf q}(t)\,{\bf X}_{l}(t),
\end{equation}
where ${\bf q}(t)\in {\mathbb H}^{1 \times n}\otimes{\mathbb R}$ is a  quaternionic row-vector function.
Differentiating (\ref{eq:left_sys_dif_nonhom1}), we have
\begin{equation}
\label{eq:dif_left_sys_dif_nonhom1}
 {\bf x}^{'}(t)=\left[{\bf q}(t)\,{\bf X}_{l}(t)\right]^{'}={\bf q}^{'}(t)\,{\bf X}_{l}(t)+{\bf q}(t){\bf X}^{'}_{l}(t).
\end{equation}
Substituting (\ref{eq:dif_left_sys_dif_nonhom1}) and (\ref{eq:left_sys_dif_nonhom1}) in  (\ref{eq:left_sys_dif1}), we obtain
\begin{equation}
\label{eq:dif_left1}
 {\bf q}^{'}(t)\,{\bf X}_{l}(t)+{\bf q}(t)\,{\bf X}^{'}_{l}(t)={\bf q}(t)\,{\bf X}_{l}(t){\bf A}(t)+{\bf b}(t).
\end{equation}
Since ${\bf X}_{l}(t)$ is a solution of the corresponding homogenous system (\ref{eq:left_sys_dif_hom}), then ${\bf X}^{'}_{l}(t)={\bf X}_{l}(t){\bf A}(t)$. Therefore, ${\bf q}^{'}(t)\,{\bf X}_{l}(t)={\bf b}(t)$ which implies
\begin{equation}
\label{eq:dif_q1}
{\bf q}^{'}(t)={\bf b}(t){\bf X}^{-1}_{l}(t).
\end{equation}
Integrating (\ref{eq:dif_q1}) over $[t_{0}, t]$, we have
\begin{equation}
\label{eq:dif_q2}
{\bf q}(t)=\int_{t_{0}}^{t}{\bf b}(s){\bf X}^{-1}_{l}(s)dt+{\bf q},
\end{equation}
where ${\bf q}\in {\mathbb H}^{1 \times n}$ is a constant quaternionic row-vector. Substituting (\ref{eq:dif_q2}) in (\ref{eq:left_sys_dif_nonhom1}), we obtain the general solution of
(\ref{eq:left_sys_dif}) representing by (\ref{eq:end_gen_left_non}).

Finally, we must verify that (\ref{eq:end_gen_left_non}) is a solution to (\ref{eq:right_sys_dif}). Differentiating (\ref{eq:end_gen_left_non}), we obtain
\begin{multline*}
[{\bf x}^{NH}_{gen}(t)]^{'}={\bf q}{\bf X}^{'}_{l}(t)+{\bf b}(t){\bf X}^{-1}_{l}(s){\bf X}_{l}(t)+ \int_{t_{0}}^{t}{\bf b}(s){\bf X}^{-1}_{l}(s)dt\,{\bf X}^{'}_{l}(t)=\\
\left({\bf q}{\bf X}_{l}(t)+\int_{t_{0}}^{t}{\bf b}(s){\bf X}^{-1}_{l}(s)dt{\bf X}_{l}(t)\right){\bf A}(t)+{\bf b}(t)={\bf x}^{NH}_{gen}(t){\bf A}(t)+{\bf b}(t).
\end{multline*}
The proof is complete.
\end{proof}

 \subsection{Quaternionic  linear systems of differential equations with constant coefficients.}
 Let ${\bf A}\in {\mathbb H}^{n \times n}$ be a constant matrix. Using properties of the exponential of a quaternion matrix and similar to the real case, we can prove the following theorems about the initial value problems with left and right quaternionic  homogeneous linear systems of differential equations.
\begin{theorem}If ${\bf A}\in {\mathbb H}^{n \times n}$, $t_{0} \in {\mathbb R}$ is an arbitrary
constant, and ${\bf x_{0}}\in {\mathbb H}^{1 \times n}$ is any constant quaternionic row-vector, then the initial value problem for the unknown
quaternionic row-vector valued function ${\bf x}$ given by
\begin{equation*}
{\bf x}'= {\bf x}{\bf A},\,\,\,{\bf x}(t_{0})={\bf x}_{0}.
\end{equation*}
has a unique solution given by the formula
\begin{equation}\label{eq:left_sys_dif_cosh_hom_sol}
{\bf x}={\bf x}_{0}e^{{\bf A}(t-t_{0})}
\end{equation}
\end{theorem}
\begin{proof} Rewrite the given  equation and multiply it on the right by $e^{-{\bf A}t}$,
\[{\bf x}^{'}e^{-{\bf A}t}-{\bf x}{\bf A}e^{-{\bf A}t}={\bf 0}\]
 Using the properties of the matrix exponential, we have
 \[\left({{\bf x}e^{-{\bf A}t}}\right)^{'}={\bf 0}.\]
 Integrating in the last equation above, and  denoting by $\bf q$ a constant row $n$-vector, we get
\[{\bf x}e^{-{\bf A}t}={\bf q}.\]
 Since  $\left({e^{-{\bf A}t}}\right)^{-1}={e^{{\bf A}t}}$, then ${\bf x}={\bf q}e^{{\bf A}t}$.
Evaluating at $t = t_0$ we get the constant vector ${\bf q}={\bf x}_{0}e^{-{\bf A}t_{0}}$ and the solution formula,
\[{\bf x}(t)
={\bf x}_{0}e^{-{\bf A}t_{0}}e^{{\bf A}t}\]
Taking account that $e^{-{\bf A}t_{0}}e^{{\bf A}t}=e^{{\bf A}(t-t_{0})}$, we finally obtain (\ref{eq:left_sys_dif_cosh_hom_sol}).

\end{proof}
The following theorem can be proved similarly.
\begin{theorem}If ${\bf A}\in {\mathbb H}^{n \times n}$, $t_{0} \in {\mathbb R}$ is an arbitrary
constant, and ${\bf x_{0}}\in {\mathbb H}^{n \times 1}$ is any constant quaternionic column-vector, then the initial value problem for the unknown
quaternionic column-vector valued function ${\bf x}$ given by
\begin{equation*}
{\bf x}'= {\bf A}{\bf x},\,\,\,{\bf x}(t_{0})={\bf x}_{0}.
\end{equation*}
has a unique solution given by the formula
\begin{equation}\label{eq:right_sys_dif_cosh_hom_sol}
{\bf x}=e^{{\bf A}(t-t_{0})}{\bf x}_{0}
\end{equation}
\end{theorem}

 \begin{remark}
Comparing both corresponding pairs of the solution formulas (\ref{eq:gen_right_init}) and (\ref{eq:right_sys_dif_cosh_hom_sol}), (\ref{eq:gen_left_init}) and (\ref{eq:left_sys_dif_cosh_hom_sol}) gives the following relation,
\begin{equation*}
e^{{\bf A}(t-t_{0})}={\bf X}_{r}(t){\bf X}_{r}^{-1}(t_{0})={\bf X}_{l}^{-1}(t_{0}){\bf X}_{l}(t)
\end{equation*}
\end{remark}

 \begin{theorem}\label{tm:A_diag_right}If ${\bf A}\in {\mathbb H}^{n \times n}$ is diagonalizable,  then the right system
$
{\bf x}^{'} = {\bf  A} {\bf x}$
 has a general solution
\begin{equation}\label{eq:gen_const}
{\bf x}_{gen}(t)={\bf v}_{1} e^{\lambda_{1} t}q_{1}+\cdots+{\bf v}_{n} e^{\lambda_{n} t}q_{n},
\end{equation}
where $\left\{{\bf v}^{(1)},\ldots,{\bf v}^{(n)}\right\}$ is a set of linearly independent column eigenvectors and corresponding eigenvalues $\left\{\lambda_{1},\ldots,\lambda_{n}\right\}$. (They can be obtained as standard eigenvalues, i.e. $\lambda_{i}\in {\mathbb C}$ for all $i=1,\ldots,n$.)  Furthermore, every initial value problem with ${\bf x}(t_{0})={\bf x}_{0}$ has a unique
solution for every initial condition ${\bf x}_{0}\in {\mathbb H}^{n \times 1}$, where the constants $q_{1},\ldots,q_{n}$ are a solution of
the  algebraic linear system
\begin{equation}\label{eq:gen_const_init}
{\bf x}_{0}={\bf v}_{1} e^{\lambda_{1} t}q_{1}+\cdots+{\bf v}_{n} e^{\lambda_{n} t}q_{n},
\end{equation}
and this solution is given by
\begin{equation}\label{eq:sol_right_const}
{\bf x}(t)={\bf X}_{r}(t){\bf X}_{r}^{-1}(t_{0}){\bf x}_{0}
\end{equation}
where ${\bf X}_{r}(t) = \left[{\bf v}_{1} e^{\lambda_{1} t}, \cdots,{\bf v}_{n} e^{\lambda_{n} t}\right]$
 is a fundamental matrix of the system.
\end{theorem}
\begin{proof}Since
the coefficient matrix ${\bf A}$ is diagonalizable, there
exist an invertible matrix ${\bf P}$ and a diagonal matrix ${\bf D}$ such that ${\bf A}={\bf P}{\bf D}{\bf P}^{-1}$.
 Introduce
this expression into the given equation  and left-multiplying it by ${\bf P}^{-1}$,
\begin{equation*}
{\bf P}^{-1}{\bf x}^{'} ={\bf P}^{-1} ({\bf P}{\bf D}{\bf P}^{-1}) {\bf x}.
\end{equation*}
 Since  ${\bf A}$ is constant, so is ${\bf P}$ and ${\bf D}$. Then, ${\bf P}^{-1}{\bf x}^{'}=({\bf P}^{-1}{\bf x})^{'} $, and
\begin{equation*}
({\bf P}^{-1}{\bf x})^{'} ={\bf D}({\bf P}^{-1} {\bf x}).
\end{equation*}
Define the new variable ${\bf y}={\bf P}^{-1} {\bf x}$. The differential equation is now given by
\begin{equation}\label{eq:Dy}
{\bf y}^{'}(t) ={\bf D}{\bf y}(t).
\end{equation}
Transform
the initial condition by ${\bf P}^{-1} {\bf x}(t_{0}) = {\bf P}^{-1} {\bf x}_{0}$, and put ${\bf y}_{0} =  {\bf P}^{-1} {\bf x}_{0}$. We get the following initial condition,  ${\bf y}(t_{0})={\bf y}_{0}$.  Then from (\ref{eq:Dy}), we have the system every equation of which can be involving by Proposition \ref{pr:lin_eq} as follows
\begin{equation*}
\begin{cases}
y_{1}^{'}(t)=\lambda_{1} y_{1}(t)\\
\vdots\\
y_{n}^{'}(t)=\lambda_{1} y_{n}(t)
\end{cases}\Rightarrow
\begin{cases}
y_{1}(t)=e^{\lambda_{1} t}q_{1}\\
\vdots\\
y_{n}(t)=e^{\lambda_{n} t}q_{n}
\end{cases}
\Rightarrow {\bf y}(t)=\begin{bmatrix}e^{\lambda_{1} t}q_{1}\\
\vdots\\
e^{\lambda_{n} t}q_{n}
\end{bmatrix}.\end{equation*}
Now, transform ${\bf y}$ back to ${\bf x}$,
\begin{equation*}
{\bf x}={\bf P} {\bf y}=\begin{bmatrix}{\bf v}_{1}& \cdots& {\bf v}_{n}\end{bmatrix}\begin{bmatrix}e^{\lambda_{1} t}q_{1}\\
\vdots\\
e^{\lambda_{n} t}q_{n}
\end{bmatrix}={\bf v}_{1} e^{\lambda_{1} t}q_{1}+\cdots+{\bf v}_{n} e^{\lambda_{n} t}q_{n}.
\end{equation*}
where ${\bf v}_{i} = {\bf v}_{.i}$ is the $i$th column of ${\bf P}$.

So, we obtain (\ref{eq:gen_const}). Evaluating it at $t_{0}$ we get (\ref{eq:gen_const_init}).

 If we choose fundamental solutions of $
{\bf x}^{'} = {\bf  A} {\bf x}$ to be
\[
\{{\bf x}_{1}(t) = {\bf v}_{1} e^{\lambda_{1} t}, \ldots, {\bf x}_{n}(t) = {\bf v}_{n} e^{\lambda_{n} t}\},
\]
then the associated fundamental matrix is ${\bf X}_{r}(t) = \left[{\bf v}_{1} e^{\lambda_{1} t}, \cdots,{\bf v}_{n} e^{\lambda_{n} t}\right]$ and the general solution can be writing as $ {\bf x} = {\bf X}_{r}(t) {\bf q}$,  where ${\bf q} =\begin{bmatrix}q_{1}\\\vdots\\q_{n}\end{bmatrix}$. Now, from the initial condition,
${\bf x}_{0} = {\bf x}(t_{0})={\bf X}_{r}(t_{0}) {\bf q},
$
follows ${\bf q}={\bf X}^{-1}_{r}(t_{0}){\bf x}_{0}$,
which makes sense, since $X_{r}(t)$ is an invertible matrix for all $t$, where it is defined.
Using this
formula for the constant vector ${\bf q}$ gives (\ref{eq:sol_right_const}).

This completes the proof.
\end{proof}
\begin{theorem}If ${\bf A}\in {\mathbb H}^{n \times n}$ is diagonalizable,  then the left system
$
{\bf x}^{'} ={\bf x} {\bf  A}$
 has a general solution,
$
{\bf x}_{gen}(t)=q_{1}e^{\lambda_{1} t}{\bf v}_{1} +\cdots+q_{n}e^{\lambda_{n} t}{\bf v}_{n},
$
where $\left\{{\bf v}^{(1)},\ldots,{\bf v}^{(n)}\right\}$ is a set of linearly independent row eigenvectors and corresponding eigenvalues $\left\{\lambda_{1},\ldots,\lambda_{n}\right\}$.  Furthermore, every initial value problem with ${\bf x}(t_{0})={\bf x}_{0}$ has a unique
solution for every initial condition ${\bf x}_{0}\in {\mathbb H}^{1 \times n}$, where the constants $q_{1},\ldots,q_{n}$ are a solution of
the  algebraic linear system,
$
{\bf x}_{0}=q_{1}e^{\lambda_{1} t}{\bf v}_{1} +\cdots+q_{n}e^{\lambda_{n} t}{\bf v}_{n},
$
and this solution  is given by
$
{\bf x}={\bf x}_{0}{\bf X}_{l}^{-1}(t_{0}){\bf X}_{l}(t),
$
where ${\bf X}_{l}(t) = \begin{bmatrix} e^{\lambda_{1} t}{\bf v}_{1} \\ \vdots\\e^{\lambda_{n} t}{\bf v}_{n} \end{bmatrix}$
 is a fundamental matrix of the system.
\end{theorem}
\begin{proof}The proof is similar to the proof of Theorem \ref{tm:A_diag_right} by using left eigenvalues of ${\bf  A}$ in the sense of (\ref{eq:left_eigen}).
\end{proof}
Now, consider the solution formulas of an initial value problems for
 nonhomogeneous right and left linear systems.
\begin{theorem}\label{tm:A_const_right}If ${\bf A}\in {\mathbb H}^{n\times n}$ is constant and the quaternionic column $n$-vector valued function ${\bf b}(t)$
 is continuous, then the initial right value problem
\begin{gather}\label{eq:A_const_right}
{\bf x}^{'}(t) = {\bf  A} {\bf x}(t)+ {\bf b}(t), \\{\bf x}(t_{0})={\bf x}_{0}\nonumber \end{gather}
has a unique solution for every initial condition $t_{0}\in {\mathbb R}$ and ${\bf x}_{0}\in {\mathbb H}^{n\times 1}$ given by
\begin{equation}\label{eq:sol_right_in_non}
{\bf x}= e^{{\bf A}(t-t_{0})}{\bf x}_{0}+ e^{{\bf A}t} \int_{t_{0}}^{t}e^{-{\bf A}\tau}  {\bf b}(\tau){\rm d}\tau.
\end{equation}
\end{theorem}
\begin{proof}The proof is similar to the proof of the same theorem in the complex or real case.

Rewrite the given equation as ${\bf x}^{'} = {\bf  A} {\bf x}+ {\bf b}$, and multiply it on the left by $e^{-{\bf A}t} $,
 \[e^{-{\bf A}t}{\bf x}^{'} - e^{-{\bf A}t}{\bf  A} {\bf x}=e^{-{\bf A}t} {\bf b}\]
 Since  $e^{-{\bf A}t}{\bf A}= {\bf A}e^{-{\bf A}t}$ and using  the formulas for the derivative of an exponential and product, then we obtain
 \[\left(e^{-{\bf A}t}{\bf x}\right)^{'} =e^{-{\bf A}t} {\bf b}\]
Integrating on the interval $[t_{0}, t]$ the last equation above gives
\[
e^{-{\bf A}t}{\bf x}(t)-e^{-{\bf A}t_{0}}{\bf x}(t_{0})= \int_{t_{0}}^{t}e^{-{\bf A}\tau}  {\bf b}(\tau){\rm d}\tau
\]
By  reorder terms and using that $\left(e^{-{\bf A}t}\right)^{-1}=e^{{\bf A}t}$, we  have
\begin{equation*}
{\bf x}(t)= e^{{\bf A}t}e^{-{\bf A}t_{0}}{\bf x}_{0}+ e^{{\bf A}t} \int_{t_{0}}^{t}e^{-{\bf A}\tau}  {\bf b}(\tau){\rm d}\tau.
\end{equation*}
Taking in account $e^{{\bf A}t}e^{-{\bf A}t_{0}}=e^{{\bf A}(t-t_{0})}$, from this we finally obtain (\ref{eq:sol_right_in_non}).
\end{proof}
\begin{theorem}If ${\bf A}\in {\mathbb H}^{n\times n}$ is constant and the quaternionic row $n$-vector valued function ${\bf b}$
 is continuous, then the initial left value problem
\begin{gather}\label{eq:A_const_left}
{\bf x}^{'}(t) =  {\bf x}(t){\bf  A}+ {\bf b}(t), \\{\bf x}(t_{0})={\bf x}_{0}\nonumber \end{gather}
has a unique solution for every initial condition $t_{0}\in {\mathbb R}$ and ${\bf x}_{0}\in {\mathbb H}^{1\times n}$ given by
\begin{equation*}
{\bf x}= {\bf x}_{0}e^{{\bf A}(t-t_{0})}+  \int_{t_{0}}^{t}  {\bf b}(\tau)e^{-{\bf A}\tau}{\rm d}\tau\,e^{{\bf A}t}.
\end{equation*}
\end{theorem}
\begin{proof}The proof is similar to the proof of Theorem \ref{tm:A_const_right}.
\end{proof}
\begin{remark}The general solutions of  the  equations (\ref{eq:A_const_right}) and (\ref{eq:A_const_left}) are,
 respectively,
\begin{gather}\label{eq:gen_right_const}
{\bf x}(t)=e^{{\bf A}t} \int e^{-{\bf A}t}  {\bf b}(t){\rm d}t,\\
{\bf x}(t)=\int {\bf b}(t)e^{-{\bf A}t}  {\rm d}t\,e^{{\bf A}t}.\nonumber
\end{gather}
\end{remark}
While there is no simple algorithm to directly calculate eigenvalues for general matrices, there are numerous special classes of matrices where eigenvalues can be directly calculated.
\begin{ex}
Consider the right  linear system
\begin{equation}\label{eq:eq2}
{\bf x}^{'}(t) =  {\bf  A}{\bf x}(t)+ {\bf b}(t),
\end{equation}
where the coefficient matrix ${\bf  A}$ from Example \ref{ex1}, i.e.
\[{\bf A}=\begin{bmatrix}1- 2.5i -0.5j+k & 4+3j+2.5k & 2- 2i -j-2.5k\\
 1.5-i -j-0.5k & 2+1.5i -3j+3k & 2+2.5i+j-k\\
                                                      0.5-i +j-0.5k & 3-i -0.5j&  1+i -1.5j-2k
\end{bmatrix},\]
and the source column-vector \[{\bf b}(t)=\begin{bmatrix}it\\-kt\\jt\end{bmatrix}.\]
By Example \ref{ex1}, there exists the invertible matrix
\[{\bf V}=\begin{bmatrix}1-0.5i+j-0.5k & j & 1+0.5i-j-0.5k\\
 i & k & i\\
                                                     -0.5+0.5i-0.5j+0.5k & 1&  -0.5+0.5i-0.5j-0.5k
\end{bmatrix},\]
such that ${\bf A}={\bf V}{\bf D}{\bf V}^{-1}$, where ${\bf D}={\rm diag}[1+i,i,3+i]$ and
\[{\bf V}^{-1}=\begin{bmatrix}-0.25+0.25i-0.25j+0.25k & -0.25-0.5i+0.75j & -0.25 -0.5i-0.25j\\
0.5i-0.5j &  -1.5 & 0.5+k\\
-0.25-0.25i+0.25j+0.25k & 0.25-0.5i+0.75j &  0.25-0.5i-0.25j
\end{bmatrix}.\]
We shall find the general solution of (\ref{eq:eq2}) by (\ref{eq:gen_right_const}). Since,
$
e^{-{\bf A}t}={\bf V}e^{-{\bf D}t}{\bf V}^{-1},
$
where
\[e^{{-\bf D}t}=\begin{bmatrix}e^{(-1-i)t} & 0 & 0\\
0 &  e^{-it} & 0\\
0 & 0 & e^{(-3-i)t}
\end{bmatrix},\]
then $${\bf x}(t)={\bf V}e^{{\bf D}t}{\bf V}^{-1}\int {\bf V}e^{-{\bf D}t}{\bf V}^{-1}  {\bf b}(t){\rm d}t={\bf V}e^{{\bf D}t}\int e^{-{\bf D}t}{\bf V}^{-1}  {\bf b}(t){\rm d}t.$$
Further,
$$
e^{{-\bf D}t}{\bf V}^{-1}{\bf b}(t)=e^{{-\bf D}t}\begin{bmatrix}(-i-0.5j)t \\
(-0.5-i+0.5j+2k)t\\
(0.5-i-k)t
\end{bmatrix}=\begin{bmatrix}e^{(-1-i)t}(-i-0.5j)t \\
 e^{-it}(-0.5-i+0.5j+2k)t\\
 e^{(-3-i)t}(0.5-i-k)t
\end{bmatrix},
$$ and by Proposition \ref{prop:e}
\begin{multline*}
\int e^{{-\bf D}t}{\bf V}^{-1}{\bf b}(t){\rm d}t=\begin{bmatrix}\int e^{(-1-i)t}t{\rm d}t (-i-0.5j) \\
\int e^{-it}t{\rm d}t (-0.5-i+0.5j+2k)\\
\int e^{(-3-i)t}t{\rm d}t (0.5-i-k)
\end{bmatrix}=\\\begin{bmatrix} (e^{(-1-i)t}\,(t(-0.5+0.5i)+0.5i)+g_{1}) (-i-0.5j) \\
 (e^{-it}\,\,(t i+1) +g_{2})(-0.5-i+0.5j+2k)\\
 (e^{(-3-i)t}\,(t(-0.3+0.1i)-0.08+0.06i))+g_{3}) (0.5-i-k)
\end{bmatrix}=\\\begin{bmatrix} e^{(-1-i)t}\,(t(0.5+0.5i+0.25j-0.25k)+0.5-0.25k)+g_{1} (-i-0.5j) \\
 e^{-it}\,\,(t (1-0.5i-2j+0.5k) -0.5-i+0.5j+2k)+g_{2}(-0.5-i+0.5j+2k)\\
 e^{(-3-i)t}\,(t(-0.05+0.35i+0.1j+0.3k)+0.02+0.11i+0.06j+0.08k)+g_{3} (0.5-i-k)
\end{bmatrix},
\end{multline*}
where $g_{n}\in {\mathbb H}$ is arbitrary for all $n=1,2,3.$

By the direct matrix multiplication, we have
\[{\bf V}e^{{\bf D}t}=\begin{bmatrix}(1-0.5i+j-0.5k)e^{(1+i)t} & j e^{it} & (1+0.5i-j-0.5k)e^{(3+i)t}\\
 ie^{(1+i)t} & k e^{it}  & ie^{(3+i)t}\\
                                                    ( -0.5+0.5i-0.5j+0.5k)e^{(1+i)t} &  e^{it} &  (-0.5+0.5i-0.5j-0.5k)e^{(3+i)t}
\end{bmatrix}.\]
If we put $g_{n}=0$ for all $n=1,2,3$, then, finally,  we obtain the following partial solution of (\ref{eq:eq2})
$$
{\bf x}(t)=\begin{bmatrix}2.4+0.7i+1.2j+0.1k \\
-1.35+2.45i-0.55j+1.35k\\
0.75-0.45i-2.25j+1.65k
\end{bmatrix}t +\begin{bmatrix}-0.06+1.57i-0.18j+0.71k \\
-2.11+0.02i-0.83j-0.44k\\
-0.6-0.57i+0.3j+2.49k
\end{bmatrix}.
$$
The correctness of the result can easily be verified by substituting it in (\ref{eq:eq2}).
\end{ex}
 \subsection{Determinantal representations of  solutions of right and left linear systems with constant coefficient matrices and sources vectors}
 \subsubsection{The case  with invertible coefficient matrices}
If ${\bf A}$ is invertible, then
 \begin{equation*}\int e^{- {\bf A}t}dt=-{\bf A}^{-1}e^{- {\bf A}t}+{\bf G},
\end{equation*}
where ${\bf G}$ is an arbitrary $n\times n$ matrix.
Then for the right nonhomogeneous system,
$
{\bf x}^{'}(t) = {\bf  A} {\bf x}(t)+ {\bf b},
$
we have the following general solution and solution of the right initial problem, respectively,
\begin{gather*}
{\bf x}(t)=e^{{\bf A}t} \int e^{-{\bf A}t}  {\rm d}t\,{\bf b}= -{\bf A}^{-1}{\bf b}+e^{{\bf A}t}{\bf G}{\bf b},\\
{\bf x}(t)=-{\bf A}^{-1}{\bf b}+e^{{\bf A}(t-t_{0})}{\bf x_{0}}{\bf b}.
\end{gather*}
For the left nonhomogeneous system,
$
{\bf x}^{'}(t) =  {\bf x}(t){\bf  A}+ {\bf b},
$
we evidently obtain the following general solution and solution of the right initial problem, respectively,
\begin{gather*}
{\bf x}(t)={\bf b}\int e^{-{\bf A}t}  {\rm d}t\,e^{{\bf A}t}= -{\bf b}{\bf A}^{-1}+{\bf b}{\bf G}e^{{\bf A}t},\\
{\bf x}(t)=-{\bf b}{\bf A}^{-1}+{\bf b}{\bf x}_{0}e^{{\bf A}(t-t_{0})}.
\end{gather*}
If ${\bf G}\equiv {\bf 0}$ or (that is equivalent) ${\bf x}_{0}\equiv {\bf 0}$, then   the partial solution of the right nonhomogeneous system, ${\bf x}(t)=-{\bf A}^{-1}{\bf b}$, ${\bf x}=\begin{bmatrix}x_{1}\\\vdots\\ x_{n}\end{bmatrix}$, due to Theorem \ref{theorem:right_system}, possess the following determinantal representations for all $i=1,\ldots,n$:
\begin{itemize}
  \item[(i)]  $x_{i}=-\frac{{\rm cdet}_{i}{\bf A}_{.i}({\bf b})}{\det{\bf A}}$  when ${\bf A}$ is Hermitian;
  \item[(ii)]  $x_{i}=-\frac{{\rm cdet}_{i}({\bf A}^{*}{\bf A})_{.i}({\bf b})}{{\rm ddet}{\bf A}}$ when ${\bf A}$ is  arbitrary.
\end{itemize}
Similarly, if ${\bf G}$ or  ${\bf x}_{0}$ is the zero matrix or the zero row-vector, respectively, then   the partial solution of the left nonhomogeneous system, ${\bf x}(t)=-{\bf b}{\bf A}^{-1}$, ${\bf x}=\begin{bmatrix}x_{1},\ldots, x_{n}\end{bmatrix}$, due to Theorem \ref{theorem:left_system}, possess the following determinantal representations for all $i=1,\ldots,n$:
\begin{itemize}
  \item[(i)]  $x_{i}=-\frac{{\rm rdet}_{i}{\bf A}_{i.}({\bf b})}{\det{\bf A}}$ when ${\bf A}$ is Hermitian;
  \item[(ii)]  $x_{i}=-\frac{{\rm rdet}_{i}({\bf A}{\bf A}^{*})_{i.}({\bf b})}{{\rm ddet}{\bf A}}$ when ${\bf A}$ is arbitrary.
\end{itemize}
 \subsubsection{The case  with non-invertible coefficient matrices}
If ${\bf A}$ is non-invertible, then due to \cite{ca1} the following theorem  can be expended to quaternion matrices.
 \begin{theorem} If ${\bf A}\in{\rm {\mathbb{H}}}^{n \times n}$ has index $k$, then
 \begin{multline}\label{eq:int_exp_sing}
 \int e^{ -{\bf A}t}{\rm d}t=\\-{\bf A}^{D}e^{- {\bf A}t}+({\bf I}-{\bf A}{\bf A}^{D})t\left[{\bf I}-\frac{{\bf A}}{2}t+\frac{{\bf A}^{2}}{3!}t^{2}+...+\frac{(-1)^{k-1}{\bf A}^{k-1}}{k!}t^{k-1}\right]+{\bf G}.  \end{multline}
\end{theorem}

\begin{proof} Differentiating Eq. (\ref{eq:int_exp_sing}) and using the series expansion
for $e^{ -{\bf A}t}$, we obtain
\begin{multline}\label{pr:int_exp}
e^{ -{\bf A}t}={\bf A}^{D}{\bf A}\left({\bf I}-{\bf A}t+\frac{{\bf A}^{2}}{2}t^{2}+\ldots+\frac{(-1)^{k}{\bf A}^{k}}{k!}t^{k}+\ldots\right)+({\bf I}-{\bf A}{\bf A}^{D})-\\({\bf I}-{\bf A}{\bf A}^{D}){\bf A}t+({\bf I}-{\bf A}{\bf A}^{D})\frac{{\bf A}^{2}}{2}t^{2}+\ldots+({\bf I}-{\bf A}{\bf A}^{D})\frac{(-1)^{k-1}{\bf A}^{k-1}}{(k-1)!}t^{k-1}.
\end{multline}
Since ${\bf A}^{D}{\bf A}={\bf A}{\bf A}^{D}$, from (\ref{pr:int_exp}) it follows the identity,
\begin{multline*}e^{ -{\bf A}t}={\bf A}^{D}{\bf A}\left({\bf I}-{\bf A}t+\frac{{\bf A}^{2}}{2}t^{2}+\ldots+\frac{(-1)^{k}{\bf A}^{k}}{k!}t^{k}+\ldots\right)-\\{\bf A}^{D}{\bf A}\left({\bf I}-{\bf A}t+\frac{{\bf A}^{2}}{2}t^{2}+\ldots+\frac{(-1)^{k}{\bf A}^{k}}{k!}t^{k}+\ldots\right)+\\
{\bf I}-{\bf A}t+\frac{{\bf A}^{2}}{2}t^{2}+\ldots+\frac{(-1)^{k}{\bf A}^{k}}{k!}t^{k}+\ldots
\end{multline*}
It completes the proof.
\end{proof}
\begin{remark}Note that from (\ref{pr:int_exp}), we have the following identity,
\begin{equation*}
e^{ -{\bf A}t}({\bf I}-{\bf A}{\bf A}^{D})=({\bf I}-{\bf A}{\bf A}^{D})\left[{\bf I}-{\bf A}t+\frac{{\bf A}^{2}}{2!}t^{2}+\cdots+\frac{(-1)^{k-1}{\bf A}^{k-1}}{(k-1)!}t^{k-1}\right].
\end{equation*}
Similarly, we have
\begin{equation}\label{eq:eA}
e^{ {\bf A}t}({\bf I}-{\bf A}{\bf A}^{D})=({\bf I}-{\bf A}{\bf A}^{D})\left[{\bf I}+{\bf A}t+\frac{{\bf A}^{2}}{2!}t^{2}+\cdots+\frac{{\bf A}^{k-1}}{(k-1)!}t^{k-1}\right].
\end{equation}
\end{remark}

Consider the right nonhomogeneous system
\begin{equation*}
{\bf x}^{'}(t) = {\bf  A} {\bf x}(t)+ {\bf b},
\end{equation*}
where ${\bf A}\in {\mathbb{H}}^{n \times n}$ is singular and $Ind\,{\bf A}=k$.
Due to Eqs. (\ref{eq:int_exp_sing}) and (\ref{eq:eA}), we have the following general solution and solution of the right initial problem, respectively,
\begin{gather}\label{eq:right_gen_dr}
{\bf x}(t)= \left\{-{\bf A}^{D}+ ({\bf I}-{\bf A}{\bf A}^{D})t \left[{\bf I}+{\bf A}t+\frac{{\bf A}^{2}}{2!}t^{2}+\cdots+\frac{{\bf A}^{k-1}}{(k-1)!}t^{k-1}\right]+e^{{\bf A}t}{\bf G}\right\}{\bf b},\\
{\bf x}(t)=\nonumber\\ \left\{-{\bf A}^{D}+ ({\bf I}-{\bf A}{\bf A}^{D})t \left[{\bf I}+{\bf A}t+\frac{{\bf A}^{2}}{2!}t^{2}+\cdots+\frac{{\bf A}^{k-1}}{(k-1)!}t^{k-1}\right]+e^{{\bf A}(t-t_{0})}{\bf x_{0}}\right\}{\bf b}.\nonumber
\end{gather}

If we put ${ \bf G}={\bf 0}$, then   the following  partial solution of (\ref{eq:right_gen_dr}) is obtained,
 \begin{equation}\label{eq:right_dif_part_sol}
                                               {\bf X}(t)=-{\bf A}^{D}{\bf b}+({\bf b}-{\bf A}^{D}{\bf A}{\bf b})t+\frac{1}{2}({\bf A}{\bf b}-{\bf A}^{D}{\bf A}^{2}{\bf b})t^{2}+...
                                                \frac{1}{k!}({\bf A}^{k-1}{\bf b}-{\bf A}^{D}{\bf A}^{k}{\bf b})t^{k}.
\end{equation}
\begin{theorem}\label{th:right_dif_part_sol}If ${\bf A}\in{\rm {\mathbb{H}}}^{n \times n}$  has index $k$ and $\rank{\rm {\bf
A}}^{k + 1} = \rank{\rm {\bf A}}^{k}=r \le n$,  then the partial solution (\ref{eq:right_dif_part_sol}), ${\bf x}(t)=(x_{i}(t))$, possess the following determinantal representation,
\begin{itemize}
  \item[(i)] when  ${\bf A}\in{\rm {\mathbb{H}}}^{n \times n}$ is Hermitian,
\begin{multline*}
  x_{i}=
  -{\frac{{{\sum\limits_{\beta \in
J_{r,\,n} {\left\{ {i} \right\}}} {{{\rm cdet}_{i} \left( {{\rm
{\bf A}}^{k+1}_{\,.\,i} \left( {\widehat{
{\bf b}}^{(k)}} \right)} \right) {\kern 1pt} _{\beta}
^{\beta} }} } }}{{{\sum\limits_{\beta \in J_{r,\,\,n}}
{{\left| {\left( {{\rm {\bf A}}^{ k+1} } \right){\kern
1pt}  _{\beta} ^{\beta} }  \right|}}} }}}+\\
 \left ({b_{i}- {\frac{{{\sum\limits_{\beta \in
J_{r,\,n} {\left\{ {i} \right\}}} {{{\rm cdet}_{i} \left( {{\rm
{\bf A}}^{k+1}_{\,.\,i} \left( {\widehat{{\rm
{\bf b}}}^{(k+1)}} \right)} \right) {\kern 1pt} _{\beta}
^{\beta} }} } }}{{{\sum\limits_{\beta \in J_{r,\,n}}
{{\left| {\left( {{\rm {\bf A}}^{k+1} } \right){\kern
1pt} _{\beta} ^{\beta} }  \right|}}} }}}} \right)t+\end{multline*}\begin{multline}\label{eq:her_left_dif_repr}
\frac{1}{2} \left ({\widehat{b}^{(1)}_{i}- {\frac{{{\sum\limits_{\beta \in
J_{r,\,n} {\left\{ {i} \right\}}} {{{\rm cdet}_{i} \left( {{\rm
{\bf A}}^{k+1}_{\,.\,i} \left( {\widehat{{\rm
{\bf b}}}^{(k+2)}} \right)} \right) {\kern 1pt} _{\beta}
^{\beta} }} } }}{{{\sum\limits_{\beta \in J_{r,\,\,n}}
{{\left| {\left( {{\rm {\bf A}}^{k+1} } \right) {\kern 1pt} _{\beta} ^{\beta} }  \right|}}} }}}} \right)t^{2}+\ldots\\
+\frac{1}{k!}
\left ({\widehat{b}^{(k-1)}_{i}- {\frac{{{\sum\limits_{\beta \in
J_{r,\,n} {\left\{ {i} \right\}}} {{{\rm cdet}_{i} \left( {{\rm
{\bf A}}^{k+1}_{\,.\,i} \left( {\widehat{{\rm
{\bf b}}}^{(2k)}} \right)} \right) {\kern 1pt} _{\beta}
^{\beta} }} } }}{{{\sum\limits_{\beta \in J_{r,\,\,n}}
{{\left| {\left( {{\rm {\bf A}}^{k+1} } \right){\kern
1pt} _{\beta} ^{\beta} }  \right|}}} }}}} \right)t^{k}
\end{multline}
where ${\rm {\bf A}}^{l}{\rm {\bf b}}=:\widehat{{{\rm
{\bf b}}}}^{(l)}
= (\widehat{b}^{(l)}_{i})\in {\mathbb{H}}^{n\times 1}$ for all $l= {k,\ldots,2k}$;

  \item[(ii)]when ${\bf A}$ is arbitrary,

 \begin{multline}\label{eq:arb_left_dif_repr}
   x_{i}=
 -{\frac{{ \sum\limits_{s = 1}^{n} {a}_{is}^{(k)}   {\sum\limits_{\beta \in J_{r,\,n} {\left\{ {s}
\right\}}} {{\rm{cdet}} _{s} \left( {\left({\bf A}^{ 2k+1} \right)^{*}\left({\bf A}^{ 2k+1} \right)_{. \,s} \left( {\widehat{{\rm
{\bf d}}}^{(0)}} \right)} \right){\kern 1pt}  _{\beta} ^{\beta} } }
}}{{{\sum\limits_{\beta \in J_{r,\,n}} {{\left| {\left({\bf A}^{ 2k+1} \right)^{*}\left({\bf A}^{ 2k+1} \right){\kern 1pt} _{\beta} ^{\beta}
}  \right|}}} }}}+\\
 \left ({b_{i}- {\frac{{\sum\limits_{s = 1}^{n} {a}_{is}^{(k)}{\sum\limits_{\beta \in
J_{r,\,n} {\left\{ {s} \right\}}} {{{\rm cdet}_{s} \left({\left({\bf A}^{ 2k+1} \right)^{*}\left({\bf A}^{ 2k+1} \right)_{. \,s} \left( {\widehat{{\rm
{\bf d}}}^{(1)}} \right)} \right) {\kern 1pt} _{\beta}
^{\beta} }} } }}{{{\sum\limits_{\beta \in J_{r,\,n}}
{{\left| {\left({\bf A}^{ 2k+1} \right)^{*}\left({\bf A}^{ 2k+1} \right){\kern 1pt} _{\beta} ^{\beta}
}  \right|}}} }}}} \right)t+\\
\frac{1}{2} \left ({\widehat{b}^{(1)}_{i}- {\frac{{\sum\limits_{s = 1}^{n} {a}_{is}^{(k)}{\sum\limits_{\beta \in
J_{r,\,n} {\left\{ {s} \right\}}} {{{\rm cdet}_{s} \left( {\left({\bf A}^{ 2k+1} \right)^{*}\left({\bf A}^{ 2k+1} \right)_{. \,s} \left( {\widehat{{\rm
{\bf d}}}^{(2)}} \right)} \right) {\kern 1pt} _{\beta}
^{\beta} }} } }}{{{\sum\limits_{\beta \in J_{r,\,n}}
{{\left| {\left({\bf A}^{ 2k+1} \right)^{*}\left({\bf A}^{ 2k+1} \right){\kern 1pt} _{\beta} ^{\beta}
} \right|}}} }}}} \right)t^{2}+\ldots+\\\frac{1}{k!}
\left ({\widehat{b}^{(k-1)}_{i}- {\frac{{\sum\limits_{s = 1}^{n} {a}_{is}^{(k)}{\sum\limits_{\beta \in
J_{r,\,n} {\left\{ {s} \right\}}} {{{\rm cdet}_{s} \left( {\left({\bf A}^{ 2k+1} \right)^{*}\left({\bf A}^{ 2k+1} \right)_{. \,s} \left( {\widehat{{\rm
{\bf d}}}^{(k)}} \right)} \right) {\kern 1pt} _{\beta}
^{\beta} }} } }}{{{\sum\limits_{\beta \in J_{r,\,n}}
{{\left| {\left({\bf A}^{ 2k+1} \right)^{*}\left({\bf A}^{ 2k+1} \right){\kern 1pt} _{\beta} ^{\beta}
}  \right|}}} }}}} \right)t^{k}
\end{multline}
where $({\bf A}^{ 2k+1})^{*} {\bf A}^{k+l}{\bf b}=:
\widehat{{
{\bf d}}}^{(l)}
= (\widehat{d}^{(l)}_{i})\in {\mathbb{H}}^{n\times 1}$ for all $l=
{0,\ldots,k}$ and
for all $i=1,\ldots,n$.
\end{itemize}
\end{theorem}
\begin{proof}\begin{itemize}
               \item[(i)]Using the determinantal representation of  ${\bf A}^{D}$ by  (\ref{eq:dr_rep_cdet}), we obtain the following determinantal representation of  ${\bf A}^{D}{\bf A}^{m}{\bf b}:=(y_{i})$,
\begin{multline*}
y_{i}=\sum\limits_{s = 1}^{n}a^{D}_{is}\sum\limits_{t = 1}^{n}a^{(m)}_{st}b_{tj}=
{\sum\limits_{\beta \in J_{r,n} {\left\{ {i} \right\}}}}\frac{{ \sum\limits_{s = 1}^{n}{{{\rm cdet}_{i}
{\left( {{\rm {\bf A}}_{.\,i}^{k+1} \left( {{\rm {\bf
a}}_{.s}}^{(k)} \right)} \right)_{\beta} ^{\beta} } }}}\cdot \sum\limits_{t = 1}^{n}a^{(m)}_{st}b_{tj}}{{\sum\limits_{\beta \in J_{r,n} } {{\left|
{\left( {{\rm {\bf A}}^{k+1} } \right)_{\beta} ^{\beta} } \right|}}}}=
\\{\sum\limits_{\beta \in J_{r,n} {\left\{ {i} \right\}}}}\frac{{ \sum\limits_{t= 1}^{n}{{{\rm cdet}_{i}
{\left( {{\rm {\bf A}}_{.\,i}^{k+1} \left( {{\rm {\bf
a}}_{.t}}^{(k+m)} \right)} \right)_{\beta} ^{\beta} } }}}\cdot b_{tj}}{{\sum\limits_{\beta \in J_{r,n} } {{\left|
{\left( {{\rm {\bf A}}^{k+1} } \right)_{\beta} ^{\beta} } \right|}}}}=
  {\frac{{{\sum\limits_{\beta \in
J_{r,\,n} {\left\{ {i} \right\}}} {{{\rm cdet}_{i} \left( {{\rm
{\bf A}}^{k+1}_{\,.\,i} \left( {\widehat{{\rm
{\bf b}}}^{(k+m)}_{.j}} \right)} \right) {\kern 1pt} _{\beta}
^{\beta} }} } }}{{{\sum\limits_{\beta \in J_{r,\,\,n}}
{{\left| {\left( {{\rm {\bf A}}^{ k+1} } \right){\kern
1pt} {\kern 1pt} _{\beta} ^{\beta} }  \right|}}} }}}
\end{multline*}
for all $i={1,\ldots,n}$ and $m={1,\ldots,k}$.
 From this, it follows (\ref{eq:her_left_dif_repr}).
               \item[(ii)]The proof of (\ref{eq:arb_left_dif_repr}) is similar to the proof of (\ref{eq:her_left_dif_repr}) by using the determinantal representation of  ${\bf A}^{D}$ by  (\ref{eq:cdet_draz}).
             \end{itemize}

\end{proof}
For the left nonhomogeneous system
\begin{equation*}
{\bf x}^{'}(t) =  {\bf x}(t){\bf  A}+ {\bf b},
\end{equation*} where ${\bf A}\in {\mathbb{H}}^{n \times n}$ is singular and $Ind\,{\bf A}=k$,
we evidently obtain the following general solution and solution of the left initial problem, respectively,
\begin{gather}\label{eq:left__gen_dr}
{\bf x}(t)={\bf b}\left\{-{\bf A}^{D}+ ({\bf I}-{\bf A}{\bf A}^{D})t \left[{\bf I}+{\bf A}t+\frac{{\bf A}^{2}}{2!}t^{2}+\cdots+\frac{{\bf A}^{k-1}}{(k-1)!}t^{k-1}\right]+{\bf Ge^{{\bf A}t}}\right\},\\
{\bf x}(t)=\nonumber\\{\bf b}\left\{-{\bf A}^{D}+ ({\bf I}-{\bf A}{\bf A}^{D})t \left[{\bf I}+{\bf A}t+\frac{{\bf A}^{2}}{2!}t^{2}+\cdots+\frac{{\bf A}^{k-1}}{(k-1)!}t^{k-1}\right]+{\bf x}_{0}e^{{\bf A}(t-t_{0})}\right\}.\nonumber
\end{gather}
If we put ${ \bf G}={\bf 0}$, then we obtain the following  partial solution of (\ref{eq:left__gen_dr}),
  \begin{equation}\label{eq:left_dif_part_sol}
                                                 {\bf X}(t)=-{\bf b}{\bf A}^{D}+({\bf b}-{\bf b}{\bf A}{\bf A}^{D})t+\frac{1}{2}({\bf b}{\bf A}-{\bf b}{\bf A}^{2}{\bf A}^{D})t^{2}+...
                                                  \frac{1}{k!}({\bf b}{\bf A}^{k-1}-{\bf b}{\bf A}^{k}{\bf A}^{D})t^{k}.
                                             \end{equation}
\begin{theorem}If ${\bf A}\in{\rm {\mathbb{H}}}^{n \times n}$  has index $k$ and $\rank{\rm {\bf
A}}^{k + 1} = \rank{\rm {\bf A}}^{k}=r \le n$,  then the partial solution (\ref{eq:left_dif_part_sol}), ${\bf X}(t)=(x_{i}(t)
)$, possess the following determinantal representation,
\begin{itemize}
  \item[(i)]when  ${\bf A}\in{\rm {\mathbb{H}}}^{n \times n}$ is Hermitian,
                                    \begin{multline*}
 x_{i}=-
  {\frac{{{\sum\limits_{\alpha \in I_{r,n}
{\left\{ {i} \right\}}} {{{\rm rdet}_{i} \left( {{\rm
{\bf A}}^{k+1}_{i\,.} \left( {\check{{\rm
{\bf b}}}^{(k)}} \right)} \right){\kern 1pt}  _{\alpha} ^{\alpha}  }} } }}{{{\sum\limits_{\alpha \in I_{r,n} }
{{\left| {\left( {{\rm {\bf A}}^{ k+1} } \right){\kern
1pt} _{\alpha} ^{\alpha}  }  \right|}}} }}}+\\
 \left ({b_{i}- {\frac{{{\sum\limits_{\alpha \in I_{r,n}
{\left\{ {i} \right\}}} {{{\rm rdet}_{j} \left( {{\rm
{\bf A}}^{k+1}_{i\,.} \left( {\check{{\rm
{\bf b}}}^{(k+1)}} \right)} \right){\kern 1pt} _{\alpha} ^{\alpha}}} } }}{{{\sum\limits_{\alpha \in I_{r,n} } {{\left| {\left( {{\rm {\bf A}}^{k+1} } \right){\kern
1pt} _{\alpha} ^{\alpha}  }  \right|}}} }}}} \right)t+ \\
\frac{1}{2} \left ({\check{b}^{(1)}_{i}- {\frac{{{\sum\limits_{\alpha \in I_{r,n}
{\left\{ {i} \right\}}} {{{\rm rdet}_{i} \left( {{\rm
{\bf A}}^{k+1}_{i\,.} \left( {\check{{\rm
{\bf b}}}^{(k+2)}} \right)} \right){\kern 1pt} _{\alpha} ^{\alpha} }} } }}{{{\sum\limits_{\alpha \in I_{r,n} } {{\left| {\left( {{\rm {\bf A}}^{k+1} } \right){\kern
1pt}  _{\alpha} ^{\alpha}  }  \right|}}} }}}} \right)t^{2}+...\\+\frac{1}{k!}
\left ({\check{b}^{(k-1)}_{ij}- {\frac{{{\sum\limits_{\alpha \in I_{r,n}
{\left\{ {i} \right\}}} {{{\rm rdet}_{i} \left( {{\rm
{\bf A}}^{k+1}_{i\,.} \left( {\check{{\rm
{\bf b}}}^{(2k)}} \right)} \right) {\kern 1pt} _{\alpha} ^{\alpha} }} } }}{{{\sum\limits_{\alpha \in I_{r,n} }
{{\left| {\left( {{\rm {\bf A}}^{k+1} } \right){\kern
1pt} {\kern 1pt} _{\alpha} ^{\alpha} }  \right|}}} }}}} \right)t^{k},
\end{multline*}
where ${\rm {\bf b}}{\rm {\bf A}}^{l}=:\check{{{\rm
{\bf b}}}}^{(l)}
= (\check{b}^{(l)}_{i})\in {\mathbb{H}}^{1\times n}$ for all $l= {k,\ldots,2k}$;
  \item[(ii)]when ${\bf A}\in{\rm {\mathbb{H}}}^{n \times n}$ is arbitrary,
 \begin{multline*}
   x_{i}=
 {\frac{\sum\limits_{s = 1}^{n}\left({{\sum\limits_{\alpha \in I_{r,\,n} {\left\{ {s}
\right\}}} {{\rm{rdet}} _{s} \left( {\left( { {\bf A}^{ 2k+1} \left({\bf A}^{ 2k+1} \right)^{*}
} \right)_{\,. s} (\check{{\rm {\bf d}}}^{(0)})} \right) {\kern 1pt} _{\alpha} ^{\alpha} } }
}\right){a}_{si}^{(k)}}{{{\sum\limits_{\alpha \in I_{r,\,n}} {{\left| {\left( { {\bf A}^{ 2k+1} \left({\bf A}^{ 2k+1} \right)^{*}
} \right){\kern 1pt} _{\alpha} ^{\alpha}
}  \right|}}} }}}+\\
 \left ({b_{ij}-  {\frac{\sum\limits_{s = 1}^{n}\left({{\sum\limits_{\alpha \in I_{r,\,n} {\left\{ {s}
\right\}}} {{\rm{rdet}} _{s} \left( {\left( { {\bf A}^{ 2k+1} \left({\bf A}^{ 2k+1} \right)^{*}
} \right)_{\,. s} (\check{{\rm {\bf d}}}^{(1)})} \right) {\kern 1pt} _{\alpha} ^{\alpha} } }
}\right){a}_{si}^{(k)}}{{{\sum\limits_{\alpha \in I_{r,\,n}} {{\left| {\left( { {\bf A}^{ 2k+1} \left({\bf A}^{ 2k+1} \right)^{*}
} \right){\kern 1pt} _{\alpha} ^{\alpha}
}  \right|}}} }}}} \right)t+\\
\frac{1}{2} \left ({\check{b}^{(1)}_{i}-  {\frac{\sum\limits_{s = 1}^{n}\left({{\sum\limits_{\alpha \in I_{r,\,n} {\left\{ {s}
\right\}}} {{\rm{rdet}} _{s} \left( {\left( { {\bf A}^{ 2k+1} \left({\bf A}^{ 2k+1} \right)^{*}
} \right)_{\,. s} (\check{{\rm {\bf d}}}^{(2)})} \right) {\kern 1pt} _{\alpha} ^{\alpha} } }
}\right){a}_{si}^{(k)}}{{{\sum\limits_{\alpha \in I_{r,\,n}} {{\left| {\left( { {\bf A}^{ 2k+1} \left({\bf A}^{ 2k+1} \right)^{*}
} \right){\kern 1pt} _{\alpha} ^{\alpha}
}  \right|}}} }}}} \right)t^{2}+...\end{multline*}\begin{equation*}+\frac{1}{k!}
\left ({\check{b}^{(k-1)}_{i}-  {\frac{\sum\limits_{s = 1}^{n}\left({{\sum\limits_{\alpha \in I_{r,\,n} {\left\{ {s}
\right\}}} {{\rm{rdet}} _{s} \left( {\left( { {\bf A}^{ 2k+1} \left({\bf A}^{ 2k+1} \right)^{*}
} \right)_{\,. s} (\check{{\rm {\bf d}}}^{(k)})} \right) {\kern 1pt} _{\alpha} ^{\alpha} } }
}\right){a}_{si}^{(k)}}{{{\sum\limits_{\alpha \in I_{r,\,n}} {{\left| {\left( { {\bf A}^{ 2k+1} \left({\bf A}^{ 2k+1} \right)^{*}
} \right){\kern 1pt} _{\alpha} ^{\alpha}
}  \right|}}} }}}} \right)t^{k},
\end{equation*}
where ${\bf b}{\bf A}^{k+l}({\bf A}^{ 2k+1})^{*} =:
\check{{{\rm
{\bf D}}}}^{(l)}
= (\check{d}^{(l)}_{ij})\in {\mathbb{H}}^{1\times n}$ for all $l=
{1,\ldots,k}$ and
for all $i={1,\ldots,n}$.
\end{itemize}
\end{theorem}
\begin{proof}The proof is similar to the proof of Theorem \ref{th:right_dif_part_sol} by using the determinantal representation of the Drazin inverse (\ref{eq:dr_rep_rdet}) for the case (i) and  (\ref{eq:rdet_draz}) for the case  (ii), respectively.
\end{proof}
\subsection{An example}
 Let
us consider the matrix equation
\begin{equation}\label{eq:ex_left_dif}
 {\bf x}'+ {\bf A}{\bf x}={\bf b},
\end{equation}
where\[{\bf A}=\begin{bmatrix}
  1 & k & -i \\
  -k & 2 & j \\
 i & -j & 1
\end{bmatrix},\,\, {\bf b}=\begin{bmatrix}
   j  \\
   -k  \\
  i
\end{bmatrix}.\]
Since ${\bf A}$ is Hermitian, ${\bf A}^{2}=\begin{bmatrix}
  3 & 4k & -3i \\
  -4k & 6 & 4j \\
 3i & -4j & 3
\end{bmatrix}$,
 $\det {\bf A}=\det {\bf A}^{2}=0$, and $\det \begin{bmatrix}
  1 & k  \\
  -k & 2
\end{bmatrix}=1$, $\det \begin{bmatrix}
  3 & 4k \\
  -4k & 6
\end{bmatrix}=2$, then $Ind\,{\bf A}=1$ and $r=\rank {\bf A}=2$. We shall find the solutions $\left(x_{i}\right)\in{\mathbb{H}}^{3\times 1}$ by (\ref{eq:her_left_dif_repr}),
\begin{equation*}
 x_{i}=
  {\frac{{{\sum\limits_{\beta \in
J_{2,\,3} {\left\{ {i} \right\}}} {{{\rm cdet}_{i} \left( {{\rm
{\bf A}}^{2}_{\,.\,i} \left( {\widehat{{\rm
{\bf b}}}^{(1)}} \right)} \right) {\kern 1pt} _{\beta}
^{\beta} }} } }}{{{\sum\limits_{\beta \in J_{2,\,3}}
{{\left| {\left( {{\rm {\bf A}}^{ 2} } \right){\kern
1pt} _{\beta} ^{\beta} }  \right|}}} }}}+
 \left ({b_{i}- {\frac{{{\sum\limits_{\beta \in
J_{2,\,3} {\left\{ {i} \right\}}} {{{\rm cdet}_{i} \left( {{\rm
{\bf A}}^{2}_{\,.\,i} \left( {\widehat{{\rm
{\bf b}}}^{(2)}} \right)} \right) {\kern 1pt} _{\beta}
^{\beta} }} } }}{{{\sum\limits_{\beta \in J_{2,\,3}}
{{\left| {\left( {{\rm {\bf A}}^{2} } \right){\kern
1pt} _{\beta} ^{\beta} }  \right|}}} }}}} \right)t\\
\end{equation*}
for all $i= {1,2,3}$. We have, ${{{\sum\limits_{\beta \in J_{2,\,3}} {{\left| {\left( {{\rm {\bf
A}}^{2}} \right) {\kern 1pt} _{\beta} ^{\beta} }
\right|}}} }}=4,$
\begin{equation*}\widehat{{
{\bf b}}}^{(1)}= {\bf A} {\bf b}=\begin{bmatrix}
  2+j  \\
   i-3k  \\
 2i+k
\end{bmatrix},\,\, \widehat{{
{\bf b}}}^{(2)}= {\bf A}^{2} {\bf b}=\begin{bmatrix}
 7+3j  \\
   4i-10k \\
 7i+3k
\end{bmatrix}.\end{equation*}
Therefore,
\begin{multline*}
x_{1}=\frac{1}{4}\left({\rm cdet}_{1}\begin{bmatrix}2+j & 4k \\
                                                      i-3k & 6
\end{bmatrix}+{\rm cdet}_{1}\begin{bmatrix}2+j & -3i \\
                                                      2i+k & 3
\end{bmatrix}\right)+\\\left(j-\frac{1}{4}\left({\rm cdet}_{1}\begin{bmatrix}7+3j & 4k \\
                                                      4i-10k & 6
\end{bmatrix}+{\rm cdet}_{1}\begin{bmatrix}7+3j & -3i \\
                                                      7i+3k & 3
\end{bmatrix}\right)\right)t=\\\frac{1}{4}(2j+0)+\left(j-\frac{1}{4}(2+2j+0)\right)t=0.5j+(-0.5+0.5j)t;\end{multline*}
\begin{multline*}
x_{2}=\frac{1}{4}\left({\rm cdet}_{2}\begin{bmatrix}3 & 2+j \\
                                                      -4k& i-3k
\end{bmatrix}+{\rm cdet}_{1}\begin{bmatrix}i-3k & 4j \\
                                                      2i+k & 3
\end{bmatrix}\right)+\\\left(-k-\frac{1}{4}\left({\rm cdet}_{2}\begin{bmatrix}3 & 7+3j \\
                                                      -4k & 4i-10k
\end{bmatrix}+{\rm cdet}_{1}\begin{bmatrix}4i-10k & 4j \\
                                                      7i+3k & 3
\end{bmatrix}\right)\right)t=-0.5i-0.5k;\end{multline*}
\begin{multline*}
x_{3}=\frac{1}{4}\left({\rm cdet}_{2}\begin{bmatrix}3 & 2+j \\
                                                      3i& 2i+k
\end{bmatrix}+{\rm cdet}_{2}\begin{bmatrix}6 & i-3k \\
                                                      -4j& 2i+k
\end{bmatrix}\right)+\\\left(i-\frac{1}{4}\left({\rm cdet}_{2}\begin{bmatrix}3 & 7+3j \\
                                                      3i& 7i+3k
\end{bmatrix}+{\rm cdet}_{1}\begin{bmatrix}6 & 4i-10k \\
                                                      -4j& 7i+3k
\end{bmatrix}\right)\right)t=\\0.5k+\left(0.5i -0.5k\right)t.
\end{multline*}
Note that we used Maple with the package CLIFFORD in the calculations.

\section{Conclusion}  A  basic theory on  first order right and left linear quaternion differential systems (LQDS) is considered in this paper. We adopted the theory of column-row determinants  for quaternion matrix   to proceed the theory of LQDS.
The algebraic structure of their general solutions are established.  Determinantal representations of solutions of  systems with constant coefficient matrices and sources vectors are obtained in both cases when coefficient matrices are invertible and singular. In the  last case, we use determinantal representations of the quaternion Drazin inverse   within the framework of the theory of column-row determinants.

The author declares that there is no conflict of interest regarding the publication of this paper.

\end{document}